\documentclass[letterpaper,12pt]{amsart}
\usepackage{graphicx}
\usepackage{amsthm, amsmath, amssymb}
\usepackage[foot]{amsaddr}
\usepackage{fullpage}
\usepackage{mathrsfs,latexsym}
\usepackage{verbatim,color}
\usepackage{enumitem,marvosym}
\usepackage[b]{esvect}
\usepackage[section]{placeins}
\usepackage{tikz}
\usetikzlibrary{shapes.multipart,shapes.geometric,topaths,calc}
\usepackage{varwidth}

\newtheorem{theorem}{Theorem}[section]
\newtheorem{base}{Base Counter}
\newtheorem{lem}[theorem]{Lemma}

\theoremstyle{remark}
\newtheorem{case}{Case}
\theoremstyle{definition}
\newtheorem*{obs}{Observation}
\newtheorem{conj}[theorem]{Conjecture}
\newtheorem{claim}{Claim}
\newtheorem*{claim4-2F}{Claim \ref{claim4}}
\newtheorem*{claim5-2F}{Claim \ref{claim5}}
\newtheorem{example}[base]{Example}
\numberwithin{equation}{section}
\newcommand{\abs}[1]{\left\lvert#1\right\rvert}

\newcommand{\ab}{\allowbreak}
\newcommand{\floor}[1]{\left\lfloor#1\right\rfloor}

\DeclareMathOperator{\poly}{poly}

\DeclareMathOperator{\pr}{PR} 
\DeclareMathOperator{\cyram}{CR} 
\DeclareMathOperator{\round}{Round}

\DeclareMathOperator{\pfq}{\varphi_{F_q}}
\DeclareMathOperator{\pcq}{\varphi_{C_q}}
\DeclareMathOperator{\prq}{\varphi_{R_q}}

\newcommand{\sH}{\ensuremath{\mathcal{H}}}

\newcommand{\pHK}{\ensuremath{\poly_{\sH}(K_n)}}
\newcommand{\pHG}{\ensuremath{\poly_{\sH}(G)}}
\renewcommand{\qedsymbol}{$\blacksquare$}
\begin{document}
%
 \title{Polychromatic colorings of $1$-regular and $2$-regular subgraphs of complete graphs}
\author{John Goldwasser$^*$}
\address{$^*$West Virginia University}
\email{jgoldwas@math.wvu.edu}
\author{Ryan Hansen$^*$}
\email{rhansen@math.wvu.edu}
%
\begin{abstract}
If $G$ is a graph and $\sH$ is a set of subgraphs of $G$, we say that an edge-coloring of $G$   is  $\sH$-polychromatic if  every graph from  $\sH$ gets all  colors present in $G$  on its edges.  The $\sH$-polychromatic number of $G$, denoted $\poly_\sH (G)$, is the largest number of colors in an $\sH$-polychromatic coloring.  In this paper we determine $\poly_\sH (G)$ exactly when $G$ is a complete graph on $n$ vertices, $q$ is a fixed nonnegative integer, and $\sH$ is one of three families: the family of all matchings spanning $n-q$ vertices, the family of all $2$-regular graphs spanning at least $n-q$ vertices, and the family of all cycles of length precisely $n-q$.  There are connections with an extension of results on Ramsey numbers for cycles in a graph.
\end{abstract}

\maketitle

\section{Introduction} 
\label{sec:introduction}
 
If $G$ is a graph and $\sH$ is a set of subgraphs of $G$, we say that an edge-coloring of $G$   is  $\sH$-{\it polychromatic } if  every graph from  $\sH$ has all  colors present in $G$  on its edges.  The $\sH$-polychromatic number of $G$, denoted $\poly_\sH (G)$ is the largest number of colors in an $\sH$-polychromatic coloring.  
If an $\sH$-polychromatic coloring of $G$ uses $\pHG$ colors,  it is called an {\it optimal } $\sH$-polychromatic coloring of $G$.

Alon \emph{et. al.} \cite{Alon:2007cd} found a lower bound for $\pHG$ when $G=Q_n$, the $n$-dimensional hypercube, and $\sH$ is the family of all subgraphs isomorphic to $Q_d$, where $d$ is fixed.  Offner \cite{Offner:2008vb} showed this lower bound is, in fact, the exact value for all $d$ and sufficiently large $n$.  Bialostocki \cite{Bialostocki:1983wo} showed that if $d=2$, then the polychromatic number is $2$ and that any optimal coloring uses each color about half the time.  Goldwasser \emph{et. al.} \cite{group_paper} considered the case when $\sH$ is the family of all subgraphs isomorphic to $Q_d$ minus an edge or $Q_d$ minus a vertex.

Bollobas \emph{et. al.} \cite{BPRS} treated the case where $G$ is a tree and $\sH$ is the set of all paths of length at least $r$, where $r$ is fixed.  Goddard and Henning \cite{Goddard:2018} considered vertex colorings of graphs such that each open neighborhood gets all colors.

For large $n$, it makes sense to consider $\pHK$ only if $\sH$ consists of sufficiently large graphs.  Indeed, if the graphs from $\sH$ have at most a fixed number $s$ of vertices, then $\pHK =1$ for sufficiently large $n$ by Ramsey's theorem, since even with only two colors there exists a monochromatic clique with $s$ vertices.

Axenovich \emph{et. al.} \cite{previous_paper} considered the case where $G=K_n$ and $\sH$ is one of three families of spanning subgraphs: perfect matchings (so $n$ must be even), $2$-regular graphs, and Hamiltonian cycles.  They determined $\pHK$ precisely for the first of these and to within a small additive constant for the other two.  In this paper, we determine the exact $\sH$-polychromatic number of $K_n$, where $q$ is a fixed nonnegative integer and $\sH$ is one of three families of graphs: matchings spanning precisely $n-q$ vertices, $(n-q)$-cycles, and $2$-regular graphs spanning at least $n-q$ vertices (so $q=0$) gives the results of Axenovich \emph{et. al.} in \cite{previous_paper} without the constant.)

This paper is organized as follows.  We give a few definitions and state the main results in Section \ref{sec:main_results}.  We give some more definitions in Section \ref{Definitions}. The optimal polychromatic colorings in this paper are all based on a type of ordering, and in Section \ref{Lemmas} we state and prove the technical ordering lemmas we will need.  In Section \ref{sec:optimal_polychromatic_colorings} we describe precisely the various ordered optimal polychromatic colorings of $K_n$.  In Section \ref{sec:proof_of_theorem:one} we prove Theorem \ref{theorem:one}, a result about matchings.  In Section \ref{sec:_c_q_polychromatic_numbers_1_and_2} we use some classical results on Ramsey numbers for cycles to take care of polychromatic numbers 1 and 2 for cycles.  In Section \ref{sec:proofs_of_theorem_and_lemmas_on_long_cycles} we prove Theorem \ref{theorem:six}, a result about coloring cycles, and use some results on long cycles in the literature to prove a lemma we need.  In Section \ref{Theorems} we give the rather long proofs of the three main lemmas we need.  In Section \ref{sec:polychromatic_cyclic_ramsey_numbers} we show how our results can be reconstituted in a context which generalizes the classical results on Ramsey numbers of cycles presented in Section \ref{sec:_c_q_polychromatic_numbers_1_and_2}.  In Section \ref{sec:conjectures_and_closing_remarks} we state a general conjecture of which most of our results are special cases.


\section{Main Results} 
\label{sec:main_results}

We call an edge coloring $\varphi$ of $K_n$  {\it ordered} if there exists an ordering $v_1,v_2,\ldots,v_n$ of $V(K_n)$ such that $\varphi(v_i v_j)=\varphi(v_i v_m)$ for all $1\leq i<j<m\leq n$. Moreover this coloring is {\it simply-ordered} if  for all $i<j<m$,   $\varphi(v_i v_m)=\varphi(v_j v_m) =a$ implies that $\varphi(v_tv_m)=a$ for all $i\leq t\leq j$.  Simply-ordered colorings play a fundamental role in this paper. An ordered edge coloring $\varphi$ induces a vertex coloring $\varphi'$ on $V(K_n)$ called the {\it $\varphi$-inherited coloring}, defined by $\varphi'(v_i)=\varphi(v_i v_m)$ for $i<m\leq n$ and $\varphi'(v_n)=\varphi'(v_{n-1})$. We can represent the induced vertex coloring $\varphi'$ by the sequence $c_1,c_2,\ldots,c_n$ of colors, where $c_i=\varphi'(v_i)$ for each $i$.
A \emph{block} in this sequence is a maximal set of consecutive vertices of the same color.  If $\varphi$ is simply-ordered then the vertices in each color class appear in a single block, so in that case, the number of blocks equals the number of colors.

Let $q$ be a fixed nonnegative integer.  We define four families of subgraphs of $K_n$ as follows.

\begin{enumerate}
	\item $F_q(n)$ is the family of all matchings in $K_n$ spanning precisely $n-q$ vertices (so $n-q$ must be even)
	\item $C_q(n)$ is the family of all cycles of length precisely $n-q$
	\item $R_q(n)$ is the family of all $2$-regular subgraphs spanning at least $n-q$ vertices.
	\item $C_q^*(n)$ is the family of all cycles of length precisely $n-q$ where $n$ and $q$ are such that $\poly_{C_q(n)}(K_n)\geq 3$.
\end{enumerate}
Further, let $\pfq(n)=\poly_{F_q(n)}(K_n)$, $\pcq(n)=\poly_{C_q(n)}(K_n)$, and $\prq(n)=\poly_{R_q(n)}(K_n)$.  Our main result is that for $F_q(n), R_q(n)$, and $C^*_q(n)$ there exist optimal polychromatic colorings which are simply ordered, or almost simply ordered (except for $C_q(n)$ if $\pcq(n)=2$).  Once we know there exists an optimal simply ordered (or nearly simply ordered) coloring, it is easy to find it and to determine a formula for the polychromatic number.  Our main results are the following.

\begin{theorem}\label{theorem:one}
	For all integers $q$ and $n$ such that $q$ is nonnegative and $n-q$ is positive and even, there exists an optimal simply-ordered $F_q$-polychromatic coloring of $K_n$.
\end{theorem}

\begin{theorem}\label{theorem:two}
	\cite{previous_paper} If $n\geq 3$, then there exist optimal $R_0$-polychromatic and $C_0$-polychromatic colorings of $K_n$ which can be obtained from simply-ordered colorings by recoloring one edge.
\end{theorem}

\begin{theorem}\label{theorem:three}
	If $n\geq 4$, then there exist optimal $R_1$-polychromatic and $C_1$-polychromatic colorings of $K_n$ which can be obtained from simply-ordered colorings by recoloring two edges.
\end{theorem}

\begin{theorem}\label{theorem:four}
	Let $q\geq 2$ be an integer.  If $n\geq q+3$, then there exists an optimal simply-ordered $R_q$-polychromatic coloring of $K_n$.  If $n\geq q+4$, then there exists an optimal simply-ordered $C_q$-polychromatic coloring except if $n\in[2q+2,3q+2]$ and $n-q$ is odd.
\end{theorem}

\begin{theorem}\label{theorem:five}
	Suppose $q\geq 2$ and $n\geq 6$
	\begin{enumerate}[label=\alph*)]
		\item If $n-q$ is even then there exists a $C_q$-polychromatic 2-coloring of $K_n$ if and only if $n\geq 3q+3$.
		\item If $n-q$ is odd then there exists a $C_q$-polychromatic 2-coloring of $K_n$ if and only if $n\geq 2q+2$.
	\end{enumerate}
\end{theorem}
Theorem \ref{theorem:five} follows from results of Bondy and Erd{\H{o}}s \cite{Bondy:1972} and Faudree and Schelp \cite{Faudree:1974}.

The following result, which is needed for the proof of Theorem \ref{theorem:four}, may be of independent interest, so we state it as a theorem:
\begin{theorem}\label{theorem:six}
	Let $n$ and $j$ be integers with $4\leq j\leq n$, and let $\varphi$ be an edge-coloring of $K_n$ with at least three colors so that every $j$-cycle gets all colors.  Then every cycle of length at least $j$ gets all colors under $\varphi$.
\end{theorem}

The statements about cycles in Theorems \ref{theorem:two}--\ref{theorem:five} can be used to get an extension of the result of Faudree and Schelp \cite{Faudree:1974} in the following manner. Let $s$ and $t$ be integers with $t\geq2, s\geq 3$, and $s\geq t$.  The $t$-polychromatic cyclic Ramsey number $\pr_t(s)$ is the smallest integer $N\geq s$ such that in any $t$-coloring of the edges of $K_N$ there exists an $s$-cycle whose edges do not contain all $t$ colors.  Note that in the special case $t=2$, this is the classical Ramsey number for cycles, the smallest integer $N$ such that in any $2$-coloring of the edges of $K_N$ there exists a monochromatic $s$-cycle.  These numbers were determined for all $s$ by Faudree and Schelp \cite{Faudree:1974}, confirming a conjecture of Bondy and Erd{\H{o}}s \cite{Bondy:1972}.

\begin{theorem}\label{extension}
	Let $\pr_t(s)$ be the smallest integer $n\geq s\geq 3$ such that in any $t$-coloring of the edges of $K_n$ there exists an $s$-cycle whose edges do not contain all $t$ colors.  If $t\geq 3$,
		\[
			\pr_t(s)=\begin{cases}
								s, & \mathrm{if\ } 3<s\leq 3\cdot 2^{t-3}\\
								s+1, & \mathrm{if\ } s\in [3\cdot2^{t-3}+1,5\cdot2^{t-2}-2]\\
								s+2, & \mathrm{if\ } s\in [5\cdot 2^{t-2}-1, 5\cdot 2^{t-1}-4]\\
								s + \round\left(\frac{s-2}{2^t-2}\right), & \mathrm{if\ } s\geq 5\cdot2^{t-1}-3
					\end{cases}
		\]
\end{theorem}
\noindent where $\round\left(\frac{s-2}{2^t-2}\right)$ is the closest integer to $\frac{s-2}{2^t - 2}$, rounding up if it is $\frac{1}{2}$ more than an integer.


\section{Definitions} 
\label{Definitions}

  Recall that if $\varphi$ is an ordered edge coloring of $K_n$ with respect to the ordering $v_1, \ldots, v_n$ of its vertices, we say that $\varphi'$ is the  {\it $\varphi$-inherited coloring} (or just {\it inherited coloring}) if it is the vertex coloring  of $K_n$ defined by $\varphi'(v_i)=\varphi(v_i v_j)$ for $1\leq i<j\leq n$ and $\varphi'(v_n)=\varphi'(v_{n-1})$.  Given an ordering of $V(K_n)$, any vertex coloring $\varphi'$ such that $\varphi'(v_{n-1})=\varphi'(v_n)$ uniquely determines a corresponding ordered coloring.  We define a \emph{color class $M_i$ of color $i$} to be the set of all vertices $v$ where $\varphi'(v)=i$.  In this paper, we shall always think of  the ordered vertices as arranged on a horizontal line with $v_i$ to the left of $v_j$ if $i<j$. We say that an edge $v_iv_m$, $i<m$  goes from $v_i$ to the right and from $v_m$ to the left.  If $X$ is a (possibly empty) subset of $V(K_n)$, we say that the edge-coloring $\varphi$ of $K_n$  is 
	\begin{itemize}
		\item \emph{$X$-constant} if  for any $v\in X$,  $\varphi(v u)=\varphi(v w)$ for all  $u, w\in V\setminus X$,
		\item \emph{$X$-ordered} if it is $X$-constant and the vertices of $X$ can be ordered $x_1, \ldots, x_m$ such that for each $i = 1,\ldots, m$,
		$\varphi(x_i x_p ) = \varphi(x_i x_m) = \varphi(x_i w)$ for all  $i<p\leq m$ and all $w\in V\setminus X$,
	\end{itemize}
If $Z$ is a nonempty subset of $V(K_n)$ we say $\varphi$ is
	\begin{itemize}
		\item \emph{$Z$-quasi-ordered} if 
			\begin{enumerate}
				\item $\varphi$ is $Z$-constant
				\item Each vertex $v_i$ in $Z$ is incident to precisely $n-2$ edges of one color, which we call the \emph{main color} of $v_i$, and one edge $v_i v_j$ of another color, where $v_j\in Z$.  If that other color is $t$, then $v_j$ is incident to precisely $n-2$ edges of color $t$.
			\end{enumerate}
	\end{itemize}
	
	It is not hard to show that there are only two possibilities for the set $Z$ in a $Z$-quasi-ordered coloring:
	\begin{enumerate}
		\item $\abs{Z}=3$, the three vertices in $Z$ have different main colors, and there is one edge in $Z$ of each of these colors
		\item $\abs{Z}=4$, with two vertices $u,v$ in $Z$ with one main color, say $i$ and two vertices $y,z$ in $Z$ with another main color, say $j$, and $\varphi(uv)=\varphi(uy)=\varphi(vz)=i,\varphi(yz)=\varphi(yv)=\varphi(zu)=j$.
	\end{enumerate}
	
	\begin{itemize}
		\item \emph{quasi-ordered} if it is $Z$-quasi-ordered and $\varphi$ restricted to $V\setminus Z$ is ordered
		\item \emph{quasi-simply ordered} if it is $Z$-quasi-ordered and $\varphi$ restricted to $V\setminus Z$ is simply ordered.
		\item \emph{nearly $X$-ordered} if it is $Z$-quasi-ordered and the restriction of $\varphi$ to $V(K_n)\setminus Z$ is $T$-ordered for some (possibly empty) subset $T$ of $V(K_n)\setminus Z$ and $X=Z\cup T$.  (If $\varphi$ is nearly $X$-ordered then one or two edges could be recolored to get an $X$-ordered coloring.)
	\end{itemize}
	
	It is easy to check that if $\varphi$ is quasi-ordered (quasi-simply-ordered) for some set $Z$ then if $\abs{Z}=3$ one edge can be recolored, and if $\abs{Z}=4$, then two edges can be recolored to get an ordered (simply-ordered) coloring.
	
	The \emph{maximum monochromatic degree} of an edge coloring of $K_n$ is the maximum number of edges of the same color incident with a single vertex.  If the maximum monochromatic degree of a coloring is $d$, and the vertex $v$ is incident with $d$ edges of color $t$, and the other $n-1-d$ edges incident with $v$ have color $s$, we say $v$ is a $t$-max vertex and also a \emph{$(t,s)$-max vertex} with \emph{majority color $t$} and \emph{minority color $s$}.
	
	We extend the notion of inherited to quasi-ordered colorings as follows.  If $\varphi$ is a quasi-ordered coloring with $\psi$ the ordered coloring which is a restriction of $\varphi$ to $V\setminus Z$, we define $\varphi'$, the $\varphi$-inherited coloring, by letting $\varphi'(x)$ equal the main color of $x$ if $x\in Z$ and $\varphi(y)=\psi'(y)$ if $y\not\in Z$.  We think of the vertices in $Z$ preceding those not in $Z$, in the order left to right, and if $\abs{Z}=4$ we list two vertices in $Z$ with the same main color first, then the other two vertices with the same main color.

 
\section{Ordering Lemmas} 
\label{Lemmas}

Let $\varphi$ be an ordered edge coloring of $K_n$ with vertex order $v_1,v_2,\ldots,v_n$, colors $1, \ldots, k$, and $\varphi'$ be the inherited coloring of $V(K_n)$.  For each $t\in[k]$ and $j\in[n]$, let $M_t$ be a color class $t$ of $\varphi'$ and $M_t(j)=M_t\cap\{v_1,v_2,\ldots,v_j\}$.  The next Lemma is a key structural lemma that characterizes ordered polychromatic colorings.

\begin{lem}\label{orderedPClemma}
	Let $\varphi:E(K_n) \to[k]$ be an ordered or quasi-ordered coloring with vertex order $v_1,v_2 \ldots, v_n$.

	Then the following statements hold:

	\begin{enumerate}[label=(\Roman*)]
		\item\label{OPC-1F}  $\varphi$ is $F_q$-polychromatic $\Longleftrightarrow$ $\forall t\in [k]$   $\exists  j\in[n]$ such that  $\abs{M_t(j)} >\frac{j+q}{2}$,
		\item\label{OPC-HC} $\varphi$ is $C_q$-polychromatic $\Longleftrightarrow$  $\forall t\in [k]$ either
		\begin{enumerate}[label=(\alph*)]
			\item\label{HCone} $\exists  j\in[q+1,n-1]$ such that $\abs{M_t(j)} \geq \frac{j+q}{2}$ or
			\item\label{HCtwo} $q=0$, $\varphi$ is $Z$-quasi-ordered with $\abs{Z}=3$ and $t$ is the color of some edge in $Z$ or
			\item\label{HCthree} $q=1$, $\varphi$ is $Z$-quasi-ordered with $\abs{Z}=4$ and $t$ is the color of some edge in $Z$.
		\end{enumerate}
		\item\label{OPC-2F} $\varphi$ is $R_q$-polychromatic  $\Longleftrightarrow$  $\forall t\in [k]$ either
		\begin{enumerate}
			\item\label{2Fone} $\exists  j\in[n]$ such that 
				\begin{enumerate}[label=(\roman*)]
					\item\label{one} $\abs{M_t(j)}>\frac{j+q}{2}$ or 
					\item\label{two} $\abs{M_t(j)}=\frac{j+q}{2}$ and $j\in\{2+q,n-2\}$ or 
					\item\label{three} $\abs{M_t(j)}=\frac{j+q}{2}$ and $\abs{M_t(j+2)}=\frac{j+q+2}{2}$ where $j\in[4+q,n-3].$
				\end{enumerate}
			\item\label{2Ftwo} $q=0$, $\varphi$ is $Z$-quasi-ordered and $t$ is the color of some edge in $Z$
			\item\label{2Fthree} $q=1$, $\varphi$ is $Z$-quasi-ordered with $\abs{Z}=3$ and $t$ is the color of some edge in $Z$
		\end{enumerate}
	\end{enumerate}
\end{lem}
	\begin{proof}
		
	Note that to prove the lemma, it is sufficient to consider an arbitrary color $t$ and show for $\sH \in \{ F_q, C_q, R_q \}$  and for each $H\in \sH$, that the given respective conditions are equivalent to $H$ containing an edge of color $t$. 
	
		
		 \emph{\ref{OPC-1F}}
		 Let $j$ be an index such that $\abs{M_t(j)}=m_j>(j+q)/2$ and let $H$ be a $1$-factor.  Let $x_1, \ldots, x_{m_j}$ be the vertices of $M_t$ in order and let $y_1,\ldots y_{j-{m_j}}$ be the other vertices of $\{v_1,v_2,\ldots,v_{j}\}$ in order.   Since $j-m_j<\frac{j-q}{2}$ and $m_j-q>\frac{j-q}{2}$, then at least one edge of $H$ with an endpoint in $M_t(j)$ must go to the right, and thus, have color $t$.
		
		On the other hand, by way of contradiction, assume that for each $j\in[n]$, $\abs{M_t(j)}\leq (j+q)/2$.  Letting $m=\abs{M_t}$, we have $m\leq (n+q)/2$.  Consider a $1$-factor that spans all vertices except for $q$ vertices in $M_t$.  Let $x_1, \ldots, x_{m-q}$ be the $m-q$ vertices remaining from $M_t$ in order and let $y_1,\ldots y_{n-m}$, be the vertices outside of $M_t$ in order.  Note that since $m\leq (j+q)/2$, it follows that $n-m\geq m-q$ since if $n-m<m-q$ then $n<2m-q$ and so $j>n$ which is impossible.  Now, let $H$ consist of the edges  $x_1y_1,  x_2y_2, \ldots, x_{m-q}y_{m-q}$ and a perfect matching on $\{y_{m-q+1}, \ldots, y_{n-m}\}$ (if this set is non-empty).  We will show that $y_i$ precedes $x_i$ in the order $v_1,v_2,\ldots,v_n$ for each $i\in[m-q]$, so $H$ has no edge of color $t$.
		
		By way of contradiction, assume $x_i$ precedes $y_i$ for some $i\in[m-q]$.  Letting $j=2i-1+q$, $y_i$ cannot be among the first $j$ vertices in the order $v_1,v_2,\ldots,v_n$, because if it were there would be at least $i+q$ vertices of color $t$ among these $j$ vertices, so a total of at least $2i+q>j$ vertices.  Hence
		\[
			\frac{j+q}{2}=\frac{2i+2q-1}{2}<i+q\leq \abs{M_t(j)}\leq \frac{j+q}{2}
		\]
		which is impossible.  Hence $y_i$ precedes $x_i$ for each $i$ and $\varphi$ is not $F_q$ polychromatic.
	%
	
	 	\emph{\ref{OPC-HC}}
		If $t$ is a color such that \ref{HCone} holds with strict inequality, the argument in \ref{OPC-1F} shows there is an edge of $H$ with color $t$.  If $\abs{M_t(j)}=\frac{j+q}{2}$ for some $j\in[q+1,n-1]$ and every edge in $H$ incident to a vertex in $M_t(j)$ goes to the left then, since each of these edges has its other vertex not in $M_t(j)$, $H$ contains $\frac{j-q}{2}$ vertices in $M_t(j)$ and the same number not in $M_t(j)$. If $\frac{j-q}{2}=1$, then the vertex in $M_t(j)$ is incident with at least one edge which goes to the right, and if $\frac{j-q}{2}\geq 2$ then $H$ contains a $2$-regular subgraph, which is impossible because an $n-q$ cycle can't have a $2$-regular subgraph on less than $n-q$ vertices.

	
	If $t$ is such that \ref{HCtwo} holds, then note that $t$ must be the main color of a vertex in $Z$ and that the cycle must contain 2 edges incident with each vertex in $Z$.  Any choice of these edges will contain an edge of color $t$ since only one edge incident with each vertex in $Z$ is not the main color of that vertex.
	
	If $t$ is such that \ref{HCthree} holds, then note that $t$ must be the main color of a vertex in $Z$ and any cycle on $n-1$ vertices must contain 2 edges incident with at least three of the four vertices in $Z$.  Any choice of these edges will contain an edge of color $t$ since only one edge incident with each vertex in $Z$ is not the main color of that vertex.
	
	On the other hand, suppose that for each $j\in [q+1,n-1]$, $\abs{M_t(j)}=m<\frac{j+q}{2}$ and $\varphi$ is not $Z$-quasi-ordered with $t$ a main color.  In particular, when $j=n-2$, we have that $\abs{M_t(j)}=m<\frac{n+q}{2}-1$.  Consider a cycle that spans all vertices except for $q$ vertices in $M_t$.  Let $x_1,\ldots,x_{m-q}$ be the other $m-q$ vertices in  $M_t$ in order and $y_1,\ldots,y_{n-m}$ be the vertices outside of $M_t$ in order.  Note that if $m<\frac{j+q}{2}$, then $n-m>m-q$ since $n-m\leq m-q \implies j>n$ which is impossible.  Consider the cycle $y_1 x_1 y_2 x_2 \cdots y_{m-q} x_{m-q} y_{m-q+1}\cdots y_{n-m} y_1$.  Suppose $y_i$ is to the right of $x_i$ for some $i$.  Then at most $i$ of the first $j=2i+q$ vertices are not in $M_t(j)$, so $\abs{M_t(j)}\geq i+q=\frac{j+q}{2}$, which is impossible.  Hence $y_i$ and $y_{i+1}$ are to the left of $x_i$ for each $1\leq i\leq m$, all edges of $H$ incident to $M_t$ go to the left, and thus are not of color $t$.
	 	
		\begin{obs} 
		If $H$ is a $2$-regular subgraph that has no edge of color $t$,  and $M$ is any subset of   $M_t$, then  all edges of $H$ incident to $M$  go to the left, so at most half the vertices in $H$ are in $M_t$ and if $\abs{M_t(j)}=\frac{j+q}{2}$, then of the first $j$ vertices, precisely $j-q$ are in $H$, precisely half of these in $M_t$, and if $j-q\geq 4$ then these $j-q$ vertices induce a 2-regular subgraph of $H$.
		\end{obs} 
	
	 	\emph{\ref{OPC-2F}}
		Let $j$ be an index such that \ref{2Fone}~\ref{one}, \ref{two}, or \ref{three} holds. Assume first that \ref{one} holds, i.e., that $\abs{M_t(j)}>\frac{j+q}{2}$ and let $H$ be a $2$-factor.  Then the argument given in \ref{OPC-1F} shows that at least one edge of $H$ with an endpoint in $M_t(j)$ must go to the right, and thus, have color $t$.  Assume that \ref{two} holds.  If $j=2+q$, then $M_t$ contains $q+1$ of the first $q+2$ vertices, so $H$ contains a vertex in $M_t$ which has an edge that goes to the right, so there is an edge of color $t$ in $H$.  If $j=n-2$ and $H$ has no edges of color $t$, then (by the previous observation) the subgraph of $H$ induced by $[n-2]$ is a $2$-factor.  Since the remaining two vertices do not form a cycle, $H$ is not a $2$-factor, a contradiction.  Finally, assume that \ref{three} holds.  If $H$ does not have an edge of color $t$, then by the previous observation, $H$ has a 2-regular subgraph spanning $j-q+2$ vertices, which has a 2-regular subgraph spanning $j-q$ vertices, which is impossible.
		
		If \ref{2Ftwo} or \ref{2Fthree} holds, by the argument for \ref{OPC-HC}, $H$ has an edge of color $t$.
	
	
	
	On the other hand, suppose that none of \ref{2Fone}, \ref{2Ftwo}, or \ref{2Fthree} hold.  We shall construct a $2$-factor that does not have an edge of color $t$.  If $\abs{M_t(j)}<\frac{j+q}{2}$ for each $j\in[q+1,n-1]$, then there is a cycle with no color $t$ edge as described in \ref{OPC-HC}.  If not, let $i_1,i_2,\ldots,i_k$ be the values of $j$ in $[4+q,n-3]$ for which $\abs{M_t(j)}=\frac{j+q}{2}$. Since \ref{2Fone}\ref{three} is not satisfied, $i_{q+1}-i_q$ is at least 4 and even for $q=1,2,\ldots,k-1$. As before, suppose there are $m$ vertices of color $t$. Let $x_1,x_2,\ldots,x_{m-q}$ be the last $m-q$ of these, in order, and let $y_1,y_2,\ldots,y_{n-m}$ be the other vertices, in order.  Note that since $m\leq \frac{n+q}{2}$ we have $m-q\leq\frac{n-q}{2}$ and $n-m\geq\frac{n-q}{2}$. For each $q$ in $[1,k-1]$, moving left to right within the interval $[i_q+1,i_{q+1}]$, there are always more $y$'s than $x$'s (except an equal number of each at the end of the interval), since otherwise there would have been another value of $j$ between $i_q$ and $i_{q+1}$ where $\abs{M_t(j)}=\frac{j+q}{2}$.  Form an $(i_{q+1}-i_q)$-cycle by alternately taking $y$'s and $x$'s, starting with the $y$ with the smallest subscript.  Also form an $i_1-q$ cycle using the first $\frac{i_1-q}{2}$ $y$'s and the same number of $x$'s, and an $n-i_k$ cycle at the end, first alternating the $y$'s and $x$'s, putting any excess $y$'s at the end.
	
	\end{proof}


\begin{lem}\label{O2SO}
Let $\sH \in \{F_q, R_q, C_q\}$. 
If there exists an ordered (quasi-ordered) $\sH$-polychromatic coloring of $K_n$ with $k$ colors, then there exists one which is simply-ordered (quasi-simply-ordered) with $k$ colors.
	\begin{proof}
		
 	Let $V(K_n) =[n]$ with the natural order.  If $c'$ is a coloring of $[n]$, a {\it block} of $c'$ is a maximal interval of integers from $[n]$ which all have the same color.  So a simply-ordered $k$-polychromatic coloring has precisely $k$ blocks.  We define a {\it block shift operation} as follows.  Assume that  $t\in[k]$ is  a color for which there are at least $2$ blocks.  Let $j(t)=j$ be the smallest integer so that $M_t(j)>(j+q)/2$ if such exists.  If there is a block $[m,s]$ in $M_t$  where $m>j$, delete this block, then  take the color of the last vertex in the remaining  sequence, and add $s-m+1$ more vertices with this color at the end of the sequence.   If   each   block  of color $t$  has its smallest  element less than or equal to $j$,  consider the block  $B$ of color $t$  that contains $j$ and consider another block $B_1$ of color $t$ that is strictly to the left of $B$.  Form a new coloring by ``moving''  $B_1$ next to $B$.  We see that the  resulting coloring has at least one less block.
 	 		
  	Let $c$ be a ordered (quasi-ordered) $F_q$-polychromatic coloring of $K_n$  on vertex set $[n]$ with $k$ colors such that the inherited vertex coloring $c'$ has the smallest possible number of blocks. Assume that color $t$  has  at least $2$ blocks.   Let $j(t)=j$ be the smallest integer so that $M_t(j)>(j+q)/2$. 
 Such $j$ exists by Lemma \ref{orderedPClemma}$\ref{OPC-1F}$, and the color of $j$ is $t$. Apply the block shifting operation.   The condition from part $\ref{OPC-1F}$ of Lemma \ref{orderedPClemma} is still valid for all color classes,  so the new coloring is $F_q$-polychromatic using $k$ colors.  This contradicts the choice of $c$ having the smallest number of blocks.
 	   
If $c$ is an ordered (quasi-ordered) $C_q$-polychromatic coloring of $K_n$, an argument very similar to the one above shows if \ref{OPC-HC}\ref{HCone}, \ref{HCtwo}, or \ref{HCthree} hold, there exists a simply-ordered (quasi-simply-ordered) coloring that uses the same number of colors and that is $C_q$-polychromatic.
 	
 Finally, let $c$ be an ordered ($X$-quasi-ordered) $R_q$-polychromatic coloring of $K_n$  on vertex set $[n]$ with $k$ colors such that the inherited vertex coloring $c'$ has the minimum possible number of blocks.  Assume that  $t\in[k]$ is  a color for which there are at least $2$ blocks. If \ref{2Ftwo} or \ref{2Fthree} hold, then the block shifting operation gives a coloring that is still $R_q$-polychromatic with the same number of colors and fewer blocks.

Thus, by Lemma \ref{orderedPClemma}\ref{OPC-2F} there exists $j$ such that 
 	\begin{enumerate}[label=($\arabic*$)]
 		\item\label{first} $\abs{M_t(j)}>(j+q)/2$ or
 		\item\label{seconda} $\abs{M_t(2+q)}=1+q$ or 
 		\item\label{secondb} $\abs{M_t(n-2)}= (n+q-2)/2$ or
 		\item\label{secondc} $\abs{M_t(n-1)}= (n+q-1)/2$ or
 		\item\label{third} $\abs{M_t(j)} = (j+q)/2$  and $\abs{M_t(j+2)} = (j+q+2)/2$ and $4+q\leq j\leq n-3.$
 	\end{enumerate}
 		
 If \ref{first} holds,  then we apply the block shifting operation and observe, as in the case of $F_q$, that the resulting coloring is still $R_q$-polychromatic with the same number of colors and fewer blocks. The case when \ref{seconda} applies is similar.

Assume neither \ref{first} nor \ref{seconda} holds.  If \ref{secondb} holds then, since $c'(v_{n-1})=c'(v_n)$, neither $v_{n-1}$ nor $v_n$ can have color $t$.  Hence there is another block of color $t$ vertices to the left of the one containing $v_{n-2}$, so we can do a block shift operation ot reduce the number of blocks, a contradiction.

The same argument works if \ref{secondc} holds.
 
 Finally, assume that none of \ref{first}--\ref{secondc} holds, but \ref{third} holds. This implies that  $c'(j)=c'(j+2)=t$ and $c'(j+1)=u\neq t$.  Now define $c''$ by $c''(i)=c'(i)$ if $i\not\in\{j+1, j+2\}, c''(j+1)=t$, and $c''(j+2)=u$.  Clearly $c''$ has at least one fewer block than $c'$. Since $j+q+1$ is odd, the only situation where $c''$ would not be $R_q$-polychromatic is if $M_u(j+1)>\frac{j+q+1}{2}$.  However, then $\abs{M_u(j-1)}=\abs{M_u(j+1)}-1>\frac{j+q-1}{2}$, so $c''$ is $R_q$-polychromatic after all.
 
	\end{proof}	
\end{lem}


\section{Optimal Polychromatic Colorings} 
\label{sec:optimal_polychromatic_colorings}

The seven following colorings are all optimal $F_q,R_q$, or $C_q$ polychromatic colorings for various values of $q$ and $n$.  Each of them is simply-ordered or quasi-simply-ordered.  We describe the color classes for each, and give a formula for the polychromatic number $k$ in terms of $q$ and $n$.

\subsection{$F_q$-polychromatic coloring $\pfq$ of $E(K_n)$ (even $n-q\geq 2$).}  
\label{subsec:_k_1f_polychromatic}
		
		Let $q$ be nonnegative and $n-q$ positive and even with $k$ a positive integer such that
		\begin{equation}\label{n_eq_F}
			(q+1)(2^k-1)\leq n<(q+1)(2^{k+1}-1).
		\end{equation}
		Let $\varphi_{F_q}$ be the simply-ordered edge $k$-coloring with colors $1,2,\ldots,k$ and inherited vertex $k$ coloring of $\varphi'_{F_q}$ with successive color classes $M_1, M_2,\ldots, M_k$, moving left to right such that $\abs{M_i}=2^{i-1}(q+1)$ if $i<k$ and $\abs{M_k}=n-\sum_{i=1}^{k-1}\abs{M_i}=n-(2^{k-1}-1)(q+1)$.  We have $k\leq \log_2\frac{n+q+1}{q+1}<k+1$ so $\pfq=k=\floor{\log_2\frac{n+q+1}{q+1}}$.


\subsection{$R_q$-polychromatic coloring $\varphi_{R_q}$ ($q\geq 2$)} 
\label{sec:_prq_polychromatic_qgeq_2_}

If $q\geq 2$, $n\geq q+3$ and $n$ and $k$ are such that \eqref{n_eq_F} is satisfied, we let $\varphi_{R_q}=\varphi_{F_q}$ (same color classes), giving us the same formula for $k$ in terms of $n$.


\subsection{$C_q$-polychromatic coloring $\varphi_{C_q}$, ($q\geq 2$).} 
\label{subsec:_k_hc_polychromatic_coloring}
		
		If $q\geq 2$, $n\geq q+3$ and
		\begin{equation}\label{n_eq_C}
			(2^{k}-1)q+2^{k-1}<n\leq (2^{k+1}-1)q+2^k
		\end{equation}
		let $\varphi_{C_q}$ be the simply-ordered edge $k$-coloring with colors $1,2,\ldots,k$ and inherited vertex $k$ coloring $\varphi'_{C_q}$ with successive color classes $M_1,M_2,\ldots,M_k$ of sizes given by:

		\begin{align*}
			\abs{M_1}&=q+1\\
			\abs{M_i}&=2^{i-1}q+2^{i-2} \rm{\ if\ }i\in[2,k-1]\\
			\abs{M_k}&=n-\sum_{i=1}^{k-1}\abs{M_i}=n-2^{k-1}q-2^{k-2}
		\end{align*}
		From equation \eqref{n_eq_C} we get $\pcq = k=\floor{\log_2\frac{2(n+q-1)}{2q+1}}$.


\subsection{$R_0$-polychromatic coloring $\varphi_{R_0}$ ($q=0$).} 
\label{subsec:_k_2f_polychromatic_coloring}

If $n\geq 3$ and $2^{k-1}-1\leq n<2^{k-1}$ let $\varphi_{R_0}$ be the quasi-simply-ordered coloring with $\abs{X}=3$ and color class sizes $\abs{M_1}=\abs{M_2}=1$ and $\abs{M_3}=n-2$ if $3\leq n\leq 6$, and if $n\geq 7$:

\begin{align*}
	\abs{M_1}&=\abs{M_2}=\abs{M_3}=1\\
	\abs{M_i}&=2^{i-2} \rm{\ if\ }i\in[4,k-1]\\
	\abs{M_k}&=n-\sum_{i=1}^k-1\abs{M_i}=n-2^{k-2}+1
\end{align*}
From this, we get $\operatorname{P_{R_0}}=k=1+\floor{\log_2(n+1)}$ where $n\geq 3$.
 

\subsection{$C_0$-polychromatic coloring $\varphi_{C_0}$ ($q=0$)} 
\label{sec:_c_0_polychromatic_q_0_}

If $n\geq 3$ and $3\cdot 2^{k-3}<n\leq 3\cdot 2^{k-2}$ let $\varphi_{C_0}$ be the quasi-simply-ordered coloring with $\abs{X}=3$ and color class sizes $\abs{M_1}=\abs{M_2}=1$ and $\abs{M_3}=n-2$ if $3\leq n\leq 6$, and if $n\geq 7$:

\begin{align*}
	\abs{M_1}&=\abs{M_2}=\abs{M_3}=1\\
	\abs{M_i}&=3\cdot 2^{i-4} \rm{\ if\ }i\in[4,k-1]\\
	\abs{M_k}&=n-\sum_{i=1}^{k-1}\abs{M_i}=n-3\cdot 2^{k-4}
\end{align*}
From this, we get $\operatorname{P_{C_0}}=k=\floor{\log_2\frac{8(n-1)}{3}}$ where $n\geq 4$.


\subsection{$R_1$-polychromatic colring $\varphi_{R_1}$ ($q=1$)} 
\label{sec:_r_1_polychromatic_q_1_}

If $n\geq 4$ and $3\cdot 2^{k-1}-2\leq n<3\cdot 2^k-2$ let $\varphi_{R_1}$ be the quasi-simply-ordered coloring with $\abs{X}=4$ and color class sizes $\abs{M_1}=2$ and $\abs{M_2}=n-2$ if $4\leq n\leq 9$, and if $n\geq 10$:

\begin{align*}
	\abs{M_1}&=\abs{M-2}=2\\
	\abs{M_i}&=3\cdot 2^{i-2} \rm{\ if\ }i\in[3,k-1]\\
	\abs{M_k}&=n-\sum_{i=1}^{k-1}\abs{M_i}=n-3\cdot2^{k-2}+2
\end{align*}
From this, we get $\operatorname{P_{R_1}}=k=\floor{\log_2\frac{2(n+2)}{3}}$ where $n\geq 4$.


\subsection{$C_1$-polychromatic coloring $\varphi_{C_1}$ ($q=1$)} 
\label{sec:_c_1_polychromatic_q_1_}

If $n\geq 4$ and $5\cdot 2^{k-2}\leq n< 5\cdot 2^{k-1}$ let $\varphi_{C_1}$ be the quasi-simply-ordered coloring with $\abs{X}=4$ and color class sizes $\abs{M_1}=\abs{M_2}=2$ and $\abs{M_3}=n-4$ if $4\leq n\leq 9$ and change every edge of color $3$ to color $2$, and if $n\geq 10$:

\begin{align*}
	\abs{M_1}&=\abs{M_2}=2\\
	\abs{M_i}&=5\cdot 2^{i-3} \rm{\ if\ }i\in[3,k-1]\\
	\abs{M_k}&=n-\sum_{i=1}^{k-1}\abs{M_i}=n-5\cdot 2^{k-3}+1
\end{align*}

From this, we get $\operatorname{P_{C_1}}=k=\floor{\log_2\frac{4n}{5}}$ where $n\geq 4$.



\section{Proof of Theorem \ref{theorem:one} on Matchings} 
\label{sec:proof_of_theorem:one}
We prove Theorem \ref{theorem:one}.  This proof is similar to the proof of Theorem 1 in \cite{previous_paper}.  Let $k=\pfq(n)$ be the polychromatic number for  $1$-factors spanning $n-q$ vertices in $G=K_n=(V,E)$.  Among all $F_q$-polychromatic colorings of $K_n$ with $k$ colors we choose ones that are $X$-ordered for a subset  $X$ (possibly empty) of the largest possible size, and, of these, choose a coloring $c$ whose restriction to $V\setminus X$ has the largest possible maximum monochromatic degree. Let $v$ be a vertex of maximum monochromatic degree, $r$,  in $c$ restricted to $G[V\setminus X]$, let the majority color on the edges incident to $v$ in $V\setminus X$ be color $1$.  By the maximality of $\abs{X}$, there is a vertex $u$ in $V\setminus X$ such that $c(uv)\neq 1$.  Assume $c(uv)=2$.  If every $1$-factor spanning $n-q$ vertices containing $uv$ had another edge of color $2$, then the color of $uv$ could be changed to $1$, resulting in a $F_q$-polychromatic coloring where $v$ has a larger maximum monochromatic degree in $V\setminus X$, a contradiction.  Hence, there is a $1$-factor $F$ spanning $n-q$ vertices in which $uv$ is the only edge with color $2$ in $c$.

Let  $c(vy_i)=1$, $y_i\in V\setminus X$, $i=1, \ldots, r$.  Note that for each $k\in[r]$, $y_k$ must be in $F$.  If not, then $F-{uv}+{vy_k}$ is a $1$-factor spanning $n-q$ vertices with no edge of color $2$ (since $uv$ was the unique edge of color $2$ in $F$ and $vy_k$ is color $1$). For each $i\in[r]$, let $y_i w_i$ be the edge of $F$ containing $y_i$ (perhaps $w_i=y_j$ for some $j\neq i$).  See Figure \ref{fig:1Fswitch}.  We can get a different $1$-factor $F_i$ by replacing the edges $uv$ and $y_i w_i$ in  $F$ with edges  $v y_i$ and $u w_i$.  Since $F_i$ must have an edge of color $2$ and $c(v y_i)=1$, we must have $c(u w_i)=2$ for each $i\in[r]$.

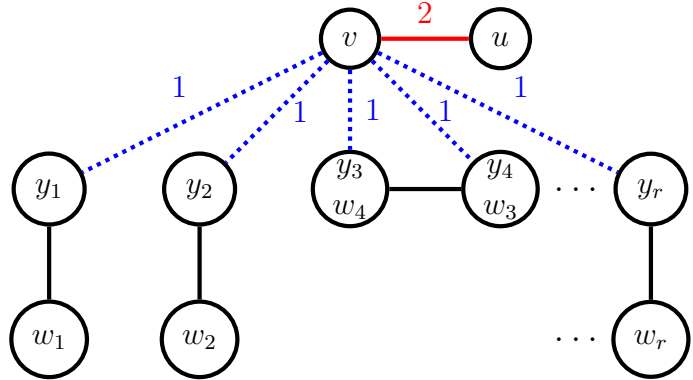
\begin{figure}[htbp]
	\centering
		\begin{tikzpicture}[every text node part/.style={align=center},scale=1,inner sep=1.75mm]
			\node[circle,ultra thick,draw=black,fill=white] (v) at (5,4) {$v$};
			\node[circle,ultra thick,draw=black,fill=white] (u) at (7,4) {$u$};
			\node[circle,ultra thick,draw=black,fill=white] (y1) at (1,2) {$y_1$};
			\node[circle,ultra thick,draw=black,fill=white] (y2) at (3,2) {$y_2$};
			\node[circle,ultra thick,draw=black,fill=white,inner sep=.3mm] (y3) at (5,2) {$y_3$\\$w_4$};
			\node[circle,ultra thick,draw=black,fill=white,inner sep=.3mm] (y4) at (7,2) {$y_4$\\$w_3$};
			\node (dots1) at (8,2) {\large $\ldots$};             
			\node[circle,ultra thick,draw=black,fill=white] (yr) at (9,2) {$y_r$};
			\node[circle,ultra thick,draw=black,fill=white] (w1) at (1,0) {$w_1$};
			\node[circle,ultra thick,draw=black,fill=white] (w2) at (3,0) {$w_2$};
			\node (dots2) at (8,0) {\large $\ldots$};
			\node[circle,ultra thick,draw=black,fill=white] (wr) at (9,0) {$w_r$};
			
			\draw[ultra thick, red] (v) -- node[above] {2} (u);
			
			\draw[ultra thick, blue,dotted] (v) -- node[above left] {1} (y1);
			\draw[ultra thick, blue,dotted] (v) -- node[right] {1} (y2);
			\draw[ultra thick, blue,dotted] (v) -- node[right] {1} (y3);
			\draw[ultra thick, blue,dotted] (v) -- node[right] {1} (y4);
			\draw[ultra thick, blue,dotted] (v) -- node[above right] {1} (yr);
			
			\draw[ultra thick, black] (y3) -- (y4);
			\foreach \i in {1,2,r} {
				\draw[ultra thick, black] (y\i) -- (w\i);
			}
		\end{tikzpicture}
	\caption{Maximum polychromatic degree in an $F_q$-polychromatic coloring}
	\label{fig:1Fswitch}
\end{figure}

If $w_i\in X$ for some $i$ then, since $c$ is $X$-constant, $c(w_iy_i) = c(w_iu) =2$,  so $y_i w_i$ and $uv$ are two edges of color $2$ in $F$, a contradiction. So, $w_i\in V\setminus X$. 
Thus $c(u v)=c(uw_1) = \cdots =c(uw_r) = 2$, and the monochromatic degree of $u$ in $V\setminus X$ is at least $r+1$, larger than that of $v$, a contradiction. 
    Hence $X=V$, $c$ is ordered, and, by Lemma \ref{O2SO}, there exists a simply-ordered $F_1$-polychromatic coloring $c_s$ with $k$ colors.  By Lemma \ref{orderedPClemma}$\ref{OPC-1F}$, if $M_1,M_2,\ldots,M_k$ are the successive color classes, moving left to right, of the inherited vertex coloring $c'_s$, then $\abs{M_t}\geq 2^{t-1}(q+1)$ for $t=1,2,\ldots,k$.  Since this inequality holds with equality for $t=1,2,\ldots,k-1$ for the inherited vertex-coloring $\pfq$, the number of color classes of $c_s$ cannot be greater than that of $\pfq$, so $k\leq \floor{\log_2 \frac{n+q+1}{q+1}}$.
\qed


\section{$C_q$-polychromatic Numbers 1 and 2} 
\label{sec:_c_q_polychromatic_numbers_1_and_2}

The following theorem is a special case of a theorem of Faudree and Schelp.

\begin{theorem}\cite{Faudree:1974}\label{FS}
	Let $s\geq 5$ be an integer and let $c(s)$ denote the smallest integer $n$ such that in any 2-coloring of the edges of $K_n$ there is a monochromatic $s$-cycle.  Then $c(s)=2s-1$ if $s$ is odd and $c(s)=\frac{3}{2}s-1$ if $s$ is even.
\end{theorem}

Faudree and Schelp actually determined all values of $c(r,s)$, the smallest integer $n$ such that in any coloring of the edges of $K_n$ with red and blue, there is either a red $r$-cycle or a blue $s$-cycle.  Their theorem extended partial results and confirmed conjectures of Bondy and Erd\H{o}s \cite{Bondy:1972} and Chartrand and Schuster \cite{ChartrandSchuster} (who showed $c(3)=c(4)=6$).  The coloring of $K_{2s-2}$ to prove the lower bound for $s$ odd is a copy of $K_{s-1,s-1}$ of red edges with all other edges blue, while for $s$ even it's a red $K_{\frac{s}{2}-1,s-1}$ with all other edges blue.

\begin{proof}[Proof of Theorem \ref{theorem:five}]

By Theorem \ref{FS}, if $s\geq 5$ is odd then there is a polychromatic 2-coloring of $K_n$ if and only if $n\leq 2s-2=2(n-q)-2$, so if and only if $n\geq 2q+2$.  If $s\geq 5$ is even then there is a polychromatic 2-coloring if and only if $n\leq \frac{3}{2}s-2=\frac{3}{2}(n-q)-2$, so if and only if $n\geq 3q+4$.  Hence if $n\in[2q+2,3q+2]$ then $\pcq(n)=1$ if $n-q$ is even and $\pcq(n)=2$ if $n-q$ is odd.  The smallest value of $n$ for which there is a simply ordered $C_q$-polychromatic 2-coloring is $n=3q+3$, so there does not exist one if $n-q$ is odd and $n\leq 3q+2$.

\end{proof}

We remark that the only values for $q\geq2$ and $n$ such that there is no optimal simply-ordered $C_q(n)$-polychromatic coloring of $K_n$ are the ones given in Theorem \ref{theorem:five} ($n\in[2q+2,3q+2]$ and $n-q$ is odd), and $q=2$, $n=5$ (two monochromatic $C_5$'s is a coloring of $K_5$ with no monochromatic $C_3$'s).


\section{Proofs of Theorem \ref{theorem:six} and Lemmas on Long Cycles} 
\label{sec:proofs_of_theorem_and_lemmas_on_long_cycles}

We will need some results on the existence of long cycles in bipartite graphs.

\begin{theorem}[Jackson \cite{Jackson:1985}]\label{Jackson}
	Let $G$ be a connected bipartite graph with bipartition $V(G)=S\cup T$ where $\abs{S}=s$, $\abs{T}=t$, and $s\leq t$.  Let $m$ be the minimum degree of a vertex in $S$ and $p$ be the minimum degree of a vertex in $T$.  Then $G$ has a cycle with length at least $\min\{2s,2(m+p-1)\}$.
\end{theorem}

\begin{theorem}[Rahman, Kaykobad, Kaykobad \cite{Rahman:2013}]\label{Rahman}
	Let $G$ be a connected $m$-regular bipartite graph with $4m$ vertices.  Then $G$ has a Hamiltonian cycle.
\end{theorem}

\begin{lem}\label{disjoint_union}
	Let $B$ be a bipartite graph with vertex bipartition $S,T$ where $\abs{S}=s$, $\abs{T}=t$, and $s\leq t$.  Suppose each vertex in $T$ has degree $m$ and each vertex in $S$ has degree $t-m$.  Then $B$ has a $2s$-cycle unless $s=t=2m$ and $B$ is the disjoint union of two copies of $K_{m,m}$.
	
	\begin{proof}
		Suppose $s<t$.  Summing degrees in $S$ and $T$ gives us $s(t-m)=tm$, so
		\[
			m=\frac{st}{s+t}>\frac{st}{2t}=\frac{s}{2}
		\]
		so $B$ is connected.  By Theorem \ref{Jackson}, $B$ has a $2s$-cycle, since $2[m+(t-m)-1]=2(t-1)\geq 2s$.  If $s=t$, then $B$ is an $m$-regular graph with $4m$ vertices.  If $B$ is connected then, by Theorem \ref{Rahman}, it has a $2s$-cycle.  If $B$ is not connected then clearly it is the disjoint union of two copies of $K_{m,m}$.
	\end{proof}
\end{lem}

We say that a cycle $H'$ of length $n-q$ is obtained  from a cycle $H$ of length $n-q$  by a {\it twist} of disjoint  edges $e_1$ and $e_2$ of $H$ if $E(H)\setminus \{e_1, e_2\} \subseteq E(H')$, i.e. we remove $e_1, e_2$ from $H$ and introduce two new edges to make the resulting graph a cycle. Note that the choice of the two edges to add is unique (due to connectedness), however, both choices would result in a $2$-regular subgraph.

One main difference between the definitions of $C_q(n)$ and $R_q(n)$ is that for the former, we consider only cycles of length precisely $n-q$, whereas, in the latter, we consider all $2$-regular subgraphs spanning \emph{at least} $n-q$ vertices.  This is because we can prove Theorem \ref{extension} for cycles, however, a similar result for $2$-regular subgraphs remains elusive (see Conjecture \ref{2-regular_conjecture}).

\subsection{Proof of Theorem \ref{theorem:six}} 
\label{subsec:proof_of_theorem_theorem_six}

		Suppose not.  Let $m$ be an integer in $[j,n-1]$ such that every $m$-cycle gets all colors but there is an $(m+1)$-cycle $H$, $v_1v_2,\ldots,v_{m+1}v_1$ which does not have an edge of color $t$.  Then $c(v_i v_{i+2})=t$ for all $i$, where the subscripts are read mod $(m+1)$, because otherwise, there is an $m$-cycle with no edge of color $t$.
		
		\begin{case}
			If $m+1$ is odd, then $v_1 v_3 v_5\cdots v_{m+1} v_2 v_4\cdots v_{m-2} v_1$ is an $m$-cycle with at most two colors, since all edges except possibly $v_{m-2} v_1$ have color $t$.  This is impossible.
		\end{case}
		
		\begin{case}
			Suppose $m+1$ is even.  Then $c_E = v_2 v_4\cdots v_{m+1} v_2$ and $c_O=v_1 v_3\cdots v_m v_1$ are $\frac{m+1}{2}$-cycles with all edges of color $t$.  Suppose $H$ has a chord $v_j v_{j+r}$ with color $t$ for some $j$ and odd integer $r$ in $[3,m-2]$.  Then $v_{j+2} v_{j+4} \cdots v_{j-2} v_j v_{j+r} v_{j+r+2}\cdots v_{j+r-4}$ is a path with $m$ vertices (missing $v_{j+r-2}$) and all edges of color $t$, so there is an $m$-cycle with at most two colors, which is impossible.  Hence if $v_i$ is a vertex in $c_E$ and $v_j$ is a vertex in $c_O$, then $v(v_i v_j)\neq t$.
			
			We claim that for each $j$ and even integer $s$, $c(v_j v_{j+s})=t$.  If not, then $v_j v_{j+s} v_{j+s+1} \cdots \ab v_{j-3} v_{j-2} v_{j+s-1} v_{j+s-2}\cdots\ab v_{j+1}v_j$ is an $m$-cycle (missing $v_{j-1}$) with no edge of color $t$ (note $c(v_{j-2}v_{j+s-1})\ab\neq t$ because $j-2$ and $j+s-1$ have different parities).  Hence, the vertices of $c_E$ and $c_O$ each induce a complete graph with $\frac{m+1}{2}$ vertices and all edges of color $t$, and there are no other edges of color $t$ in $K_n$.
			
			If there is a color $w$, different than $t$, such that there exist two disjoint edges of color $w$, then it is easy to find an $m$-cycle with two edges of color $w$ and the rest of color $t$.  If there do not exist two such edges of color $w$, then all edges of color $w$ are incident to a single vertex $x$, so any $m$-cycle with $x$ incident to two edges of color $t$ does not contain an edge of color $w$ (these exist since $\frac{m+1}{2}\geq 3$).\hfill\qedsymbol
		\end{case}


We remark that the statement in Theorem \ref{theorem:six} would be false without the requirement that there be at least three colors.  If $m\geq 3$ is odd, then two vertex disjoint complete graphs each with $\frac{m+1}{2}$ vertices and all edges of color $t$ with all edges between them of color $w$ has an $(m+1)$-cycle with all edges of color $w$, while every $m$-cycle has edges of both colors.  This is the reason for the difference between odd and even values of $n-q$ in Theorem \ref{theorem:five}.  The statement would also be false with three colors if $j=3$ and $n=4$.


\section{Main Lemmas and Proofs of Theorems} 
\label{Theorems}

We now state and prove the three main lemmas needed for the proofs of Theorems \ref{theorem:two}, \ref{theorem:three}, and \ref{theorem:four}.

		\begin{lem}\label{max-vertex}
			\hfill
			\begin{enumerate}[label=(\alph*)]
				\item\label{mvert1} Let $\sH\in\{R_q(n),C^*_q(n)\}$.  Of all optimal $\sH$-polychromatic colorings, let $\varphi$ be one which is $X$-ordered on a (possibly empty) subset $X$ of $V(K_n)$ of maximum size and, of these, such that $G_m=K_n[Y]$ has a vertex $v\in Y$ of maximum possible monochromatic degree $d$ in $G_m$ where $Y=V(K_n)\setminus X$, $\abs{Y}=m$, and $d<(m-1)$.  If $v$ is incident in $G_m$ to $d$ edges of color 1 and $u\in Y$ is such that $\varphi(vu)=2$, then $v$ is a $(1,2)$-max vertex in $G_m$ and $u$ is a $(2,t)$-max vertex in $G_m$ for some color $t$ (possibly $t=1$).
				\item\label{mvert2} The same is true if $X\neq \emptyset$ and $\varphi$ is nearly $X$-ordered.
			\end{enumerate}

			\begin{proof}[Proof of \ref{mvert1}]
				Let $y_1,y_2,\ldots,y_d\in Y$ be such that $\varphi(vy_i)=1$.  Let $H\in C_q^*$ or $H\in R_q$ be such that $uv$ is the only edge of color 2.  There must be such an $H$ otherwise we could change the color of $uv$ from $2$ to $1$, giving an $\sH$-polychromatic coloring with monochromatic degree greater than $d$ in $G_m$. \Lightning.  Orient the edges of $H$ to get a directed cycle or $2$-regular graph $H'$ where $\vv{uv}$ is an arc.

				If $y_i\in H'$ then the predecessor $w_i$ of $y_i$ in $H'$ must be such that $\varphi(w_i u)=2$, because otherwise we can twist $uv$ and $w_i y_i$ to get an $(n-q)$-cycle (if $H\in C_q^*$) or a $2$-regular graph (if $H\in R_q$) with no edge of color 2. Note that $w_i$ must be in $Y$ because otherwise, since $\varphi$ is $X$-constant, $\varphi(w_i u)=\varphi(w_i y_i)=2$, contradicting the assumption that $uv$ is the only edge in $H$ of color 2.

				Suppose $y_i\not\in H$ for some $i\in[d]$.  If $\varphi(y_i u)\neq 2$, then $J=(H\setminus\{uv\})\cup\{vy_i,y_iu\}$ has no edge of color $2$.  This is impossible if $H\in R_q$, because $J$ is a $2$-regular graph spanning $n-q+1$ vertices.  If $H\in C_q^*$, then $J$ is an $(n-q+1)$-cycle with no edge of color $2$, so by Theorem \ref{theorem:six}, since the polychromatic number of $H$ is at least $3$, there exists an $(n-q)$-cycle which is not polychromatic, a contradiction.  Hence $\varphi(y_i u)=2$ in either case.

				Thus, for each $i\in[d]$, either $y_i\not\in H$ and $\varphi(y_i u)=2$, or $y_i\in H$ and $\varphi(w_i u)=2$ where $w_i$ is the predecessor of $y_i$ in $H'$.  That gives us $d$ edges in $G_m$ of color $2$ which are incident to $u$.  Since $v$ has maximum monochromatic degree in $G_m$, it follows that $v=w_i$ for some $i$ (otherwise $uv$ is a different edge of color $2$ incident to $u$) and it also follows that no edge in $G_m$ incident to $v$ can have color $t$ where $t\not\in\{1,2\}$.  This is because if $vz$ were such an edge, as shown above, then either $z\in H$ and $\varphi(w'u=2)$ where $w'$ is the predecessor of $z$ in $H'$, or $z\not\in H$ and $\varphi(zu)=2$.  In either case we get $d+1$ edges of color $2$ in $G_m$ incident to $u$, a contradiction.  So $v$ is a $(1,2)$-max-vertex and $u$ is a $(2,t)$-max-vertex for some color $t$.
				
				The proof of \ref{mvert2} is exactly the same.
			\end{proof}
		\end{lem}

\begin{lem}\label{structurelemma}Let $n\geq 7$ and $\sH\in\{R_q(n),C_q(n)\}$.  If there does not exist an optimal {\sH}-polychromatic coloring of $K_n$ with maximum monochromatic degree $n-1$, then one of the following holds.
		\begin{enumerate}[label=\alph*)]
			\item\label{structure-one} $\sH=C_q(n)$, $n-q$ is odd and $n\in[2q+2,3q+2]$ (and $\pcq(n)=2$).
			\item $q=0$ and there exists an optimal $\sH$-polychromatic coloring which is $Z$-quasi-ordered with $\abs{Z}=3$.
			\item $q=1$ and there exists an optimal $\sH$-polychromatic coloring which is $Z$-quasi-ordered with $\abs{Z}=4$.
		\end{enumerate}
	\begin{proof}\let\qed\relax
		First assume that $\sH=C_q(n)$ and that $q\geq 2$ and $n$ are such that $\pcq(n)\leq 2$.  If $n-q$ is even then, by Theorem \ref{theorem:five}, there is a $C_q$-polychromatic 2-coloring if and only if $n\geq 3q+3$.  Since $3q+3$ is the smallest value of $n$ such that the simply-ordered $C_q$-polychromatic coloring $\varphi_{C_q}$ uses two colors, if $\pcq(n)\leq 2$ and $n-q$ is even, then there is an optimal simply-ordered $C_q$-polychromatic coloring, and this coloring has a vertex (in fact $q+1$ of them) with monochromatic degree $n-1$.
		
		If $n-q$ is odd then, by Theorem \ref{theorem:five}, there is a $C_q$-polychromatic 2-coloring if and only if $n\geq 2q+2$.  Since there is a simply-ordered $C_q$-polychromatic 2-coloring if $n\geq 3q+3$, that means that if $n-q$ is odd, $\pcq(n)\leq 2$ and $n\not\in[2q+2,3q+2]$ then there is a simply-ordered $C_q$-polychromatic coloring.  Thus if $\pcq(n)\leq 2$, there is an optimal simply-ordered $C_q$-polychromatic coloring, and hence one with maximum monochromatic degree $n-1$, unless $n-q$ is odd and $n\in[2q+2,3q+2]$, which are the conditions for \ref{structure-one}.
		
		Now let $\sH\in\{R_q(n),C^*_q(n)\}$ and suppose there does not exist an optimal $\sH$-polychromatic coloring of $K_n$ with maximum monochromatic degree $n-1$.  Of all optimal $\sH$-polychromatic colorings of $K_n$, let $\varphi$ be the one with maximum possible monochromatic degree $d$ (so $d<n-1$).
\end{proof}
\end{lem}
		\begin{claim}\label{dsize}
			$d>\frac{n-1}{2}$.
			\begin{proof}\let\qed\relax
				Since there are only two colors at a max-vertex, certainly $d\geq \frac{n-1}{2}$.  Assume $d=\frac{n-1}{2}$ (so n is odd) and that $x$ is a max-vertex where colors $i$ and $j$ appear.  Then $x$ is both an $i$-max and $j$-max vertex so, by Lemma \ref{max-vertex}, each vertex in $V$ is a max-vertex.

				Suppose there are more than 3 colors, say colors $i,j,s,t$ are all used.  If $i$ and $j$ appear at $x$ then no vertex $y$ can have colors $s$ and $t$, because there is no color for $xy$.  So the sets of colors on the vertices is an intersecting family of $2$-sets.  Since there are at least 4 colors, the only way this can happen is if some color, say $i$, appears at every vertex.  Let $n_{ij}, n_{is}$, and $n_{it}$ be the number of $(i,j)$-max, $(i,s)$-max, and $(i,t)$-max vertices with $n_{ij}\leq n_{is} \leq n_{it}$.  Then $n_{ij}<\frac{n}{2}$ (in fact, $n_{ij}\leq \frac{n}{3}$).  If $x$ is an $(i,j)$-max vertex and $y$ is an $(i,s)$-max vertex, then $c(xy)=i$.  Hence the number of edges of color $j$ incident to $x$ is at most $n_{ij}-1<\frac{n-2}{2}<d$, a contradiction.
			
				Now suppose there are precisely 3 colors.  Let $A, B, C$ be the set of all $(1,2)$-max, $(2,3)$-max, and $(1,3)$-max vertices, respectively, with $\abs{A}=a, \abs{B}=b$, and $\abs{C}=c$.  All edges from a vertex in $A$ to a vertex in $B$ have color 2, from $B$ to $C$ have color 3, from $A$ to $C$ have color 1; internal edges in $A$ have color 1 or 2, in $B$ have color 2 or 3, in $C$ have color 1 or 3.  We clearly cannot have $a,b,$ or $c$ greater than $\frac{n-1}{2}$ so, without loss of generality, we can assume $a\leq b\leq c\leq \frac{n-1}{2}$ and $a+b+c=n$.
			
				Consider the graph $F$ formed by the edges of color 1 or 2.  Vertices of $F$ in $B$ or $C$ have degree $\frac{n-1}{2}$, while vertices in $A$ have degree $n-1$. Since $a\leq c$ we have $a\leq \frac{n-b}{2}$. The internal degree in $F$ of each vertex in $B$ is $\frac{n-1}{2}-a\geq \frac{n-1}{2}-\frac{n-b}{2}=\frac{b-1}{2}$.  As is well known (Dirac's theorem), that means there is a Hamiltonian path within $B$.  Similarly there is one within $C$.  If $a\geq 2$, that makes it easy to construct a Hamiltonian cycle in $F$.  If $a=1$ we must have $b=c=\frac{n-1}{2}$, so $F$ is two complete graphs of size $\frac{n+1}{2}$ which share one vertex.  This graph has a spanning 2-regular subgraph if $n\geq 7$ (a 3-cycle and a 4-cycle if $n=7$), so no $R_q$-polychromatic coloring with 3 colors for any $q\geq 0$ if $n\geq 7$.
			
				If $a=1$ and $b=c=\frac{n-1}{2}$ consider the subgraph of all edges of colors 1 or 3.  It consists of a complete bipartite graph with vertex parts $A\cup B$ and $C$, with sizes $\frac{n+1}{2}$ and $\frac{n-1}{2}$, plus internal edges in $C$.  Clearly this graph has an $(n-1)$-cycle, but no Hamiltonian cycle.  Hence there can be a $C_q$-polychromatic 3-coloring only if $q=0$.  However, the $C_0$-polychromatic coloring $\varphi_{C_0}$ uses at least 4 colors if $n\geq 7$, so there is no optimal one with maximum monochromatic degree $\frac{n-1}{2}$.
			\end{proof}
		\end{claim}

		\begin{claim}
			If $q=0$, then, up to relabeling the colors, there is a $(1,2)$-max-vertex, a $(2,3)$-max-vertex and a $(3,1)$-max-vertex.
			\begin{proof}\let\qed\relax
				Assume that every max-vertex has majority color either $1$ or $2$.  Then $u$ must be a $(2,1)$-max-vertex.  This is because by Lemma \ref{max-vertex}, if it were a $(2,t)$-max-vertex for some third color $t$, and $c(u z)=t$, then $z$ would have to be a $t$-max-vertex, a contradiction.  Hence, every max-vertex is either a $(1,2)$-max-vertex or a $(2,1)$-max-vertex.  Let $S$ be the set of all $(1,2)$-max-vertices, $T$ be the set of all $(2,1)$-max-vertices, and $W=V\setminus(S\cup T)$.  Edges within $S$ and from $S$ to $W$ must have  color $1$ (because any minority color edge at a max-vertex is incident to a max-vertex of that color), edges within $T$ and from $T$ to $W$ must have color $2$, and all edges between $S$ and $T$ must have color $1$ or $2$.  If $\abs{S}=s$ and $\abs{T}=t$ and $m=n-1-d$, then each vertex in $S$ is adjacent to $m$ vertices in $T$ by edges of color $2$ (and adjacent to $t-m$ vertices in $T$ by edges of color $1$), and each vertex in $T$ is adjacent to $m$ vertices in $S$ by edges of color $1$.

				Suppose $s<t$ and consider any edge $ab$ from $S$ to $T$ of color 2.  As before, there is an $H\in\sH$ which contains $ab$, but no other edges of color 2.  Hence $H$ has no edges from $T$ to $W$.  Since $s<t$ there must be an edge of $H$ with both vertices in $T$, so it does have another edge of color 2 after all, a contradiction.  The same argument works if $t<s$ with an edge with color 1.  To avoid this, we must have $s=t=2m$.  If there is an edge from $S$ to $W$ then, again, $H$ has an internal edge in $T$, which is impossible.  Hence if $\sH=C^*_0$ then $W=\emptyset$ and every edge has color 1 or 2, which is impossible since $H$ has at least 3 colors.  If $\sH=R_0$ then the subgraph of $H$ induced by $S\cup T$ is the union of cycles.  If $m=1$ then $S\cup T$ induces a 4-cycle in $H$, two edges of each color, so $ab$ is not the only edge with color 2.  If $m\geq 2$ then two applications of Hall's Theorem gives two disjoint perfect matchings of edges of color 1 between $S$ and $T$, whose union is a 2-factor of edges of color 1 spanning $S\cup T$, which together with the subgraph of $H$ induced by $W$, produces a 2-factor $H'\in R_0$ with no edge of color 2.

				We have shown that $u$ is not a $(2,1)$-max vertex, so it must be a $(2,3)$-max vertex for some other color 3.  Say $\varphi(uz)=3$.  Then, by Lemma \ref{max-vertex}, $z$ is a $3$-max vertex.  If $\varphi(vz)=2$, then $z$ would be a $2$-max vertex. So $z$ would be both a $2$-max and a $3$-max vertex, and so $d=\frac{n-1}{2}$, a contradiction to Claim \ref{dsize}.  Hence $\varphi(vz)=1$, which means $z$ must be a $(3,1)$-max vertex.
			\end{proof}
		\end{claim}

		\begin{claim}\label{structure}
			If $q=0$ then $V$ can be partitioned into sets $A,B,D,E$ where the following properties hold (see Figure \ref{fig:graphfigure}).
				\begin{enumerate}
					\item\label{C5one} All vertices in $A$ are $(1,2)$-max-vertices.
					\item\label{C5two} All vertices in $B$ are $(2,3)$-max-vertices.
					\item\label{C5three} All vertices in $D$ are $(3,1)$-max-vertices.
					\item\label{C5four} No vertex in $E$ is a max-vertex.
					\item\label{C5five} All edges within $A$, from $A$ to $D$, and from $A$ to $E$ are color 1.
					\item\label{C5six} All edges within $B$, from $B$ to $A$, and from $B$ to $E$ are color 2.
					\item\label{C5seven} All edges within $D$, from $D$ to $B$, and from $D$ to $E$ are color 3.
					\item\label{C5eight} $\abs{A}=\abs{B}=\abs{D}=m=n-1-d$.
				\end{enumerate}

				\begin{figure}[htbp]
					\centering
					\begin{tikzpicture}[every text node part/.style={align=center},scale=3,inner sep=1mm]
						\node[circle,ultra thick,draw=black,fill=white,inner sep=6.5mm] (e) at (0,0) {$E$};
						\node[circle,ultra thick,draw=green!50!black,fill=white] (d) at (0,1.1547) {$D$\\$(3,1)$-max};
						\node[circle,ultra thick,draw=blue,fill=white] (a) at (-1,1.73205) {$A$\\$(1,2)$-max};
						\node[circle,ultra thick,draw=red,fill=white] (b) at (1,1.73205) {$B$\\$(2,3)$-max};
						\draw [ultra thick,blue] (a) .. controls (-1.85,1.85) and (-1.4,2.5) .. (a) {node [above left,pos=.5] {\large 1}};
						\draw [ultra thick,loosely dashed,red] (b) .. controls (1.4,2.5) and (1.85,1.85) .. (b) {node [above right,pos=.5] {\large 2}};
						\draw [ultra thick,dotted,green!50!black] (d) .. controls (-.5,2) and (.5,2) .. (d) {node [below,pos=.5] {\large 3}};
						\draw [ultra thick,blue] (a) -- node[below] {\large 1} (d);
						\path (a) edge [ultra thick,blue,bend right] node[below left] {\large 1} (e);
						\path (b) edge [ultra thick,loosely dashed,red,bend right] node[above] {\large 2} (a);
						\path (b) edge [ultra thick,loosely dashed,red,bend left] node[below right] {\large 2} (e);
						\draw [ultra thick,dotted,green!50!black] (d) -- node[below] {\large 3} (b);
						\draw [ultra thick,dotted,green!50!black] (d) -- node[left] {\large 3} (e);
					\end{tikzpicture}
					\caption{}
					\label{fig:graphfigure}
				\end{figure}
			\begin{proof}\let\qed\relax
				Let $A=\{x : x\textrm{ is a }(1,2)\textrm{-max vertex}\}$, $B=\{x : x\textrm{ is a }(2,3)\textrm{-max vertex}\}$, $D=\{x : x\textrm{ is a }(3,1)\textrm{-max vertex}\}$ and $E=V\setminus (A\cup B \cup D)$.  Let $x\in A$.  If $y\in A$, then $\varphi(xy)=1$ because if $\varphi(xy)=2$, then $y$ would be a $2$-max vertex.  If $y\in B$, then $\varphi(xy)=2$ because that is the only possible color for an edge incident to $x$ and $y$ and, similarly, if $y\in D$, then $\varphi(xy)=1$.

				Suppose $w$ is a max-vertex in $E$.  Then the two colors on edges incident to $w$ must be a subset of $\{1,2,3\}$, because, otherwise, it would be disjoint from $\{1,2\}$, $\{2,3\}$, or $\{1,3\}$, so there would be an edge incident to $w$ for which there is no color.  Say $1$ and $2$ are the colors at $w$.  Since $w\not\in A$, $w$ is a $(2,1)$-max vertex.  Let $z$ be a $(3,1)$-max vertex.  Then the edge $wz$ must have color 1 so, by Lemma \ref{max-vertex}, $z$ is a $1$-max vertex, a contradiction.  We have now verified \eqref{C5one}--\eqref{C5four}.  If $x\in A$ and $w\in E$ then $\varphi(xw)=1$ because if $\varphi(xw)=2$ then $w$ would be a $2$-max vertex.  Similar arguments show that if $y\in B$ then $\varphi(yw)=2$ and if $y\in D$ then $\varphi(yw)=3$.  We have now verified \eqref{C5one}--\eqref{C5seven}.

				We have shown that if $x$ is in $A$ then $\varphi(xy)=2$ if and only if $y\in B$.  That means $\abs{B}=m$, and by the same argument $\abs{A}=\abs{C}=m$ as well, completing the proof of Claim \ref{structure}.
			\end{proof}
		\end{claim}

		\begin{claim}\label{q=0_optimal-quasi-ordered}
			If $\sH\in\{C_0^*,R_0\}$, and there exists an optimal $\sH$-polychromatic coloring satisfying \eqref{C5one}--\eqref{C5eight} with $m>1$, then there exists one with $m=1$, i.e. one that is $Z$-quasi-ordered with $\abs{Z}=3$.
			\begin{proof}
				Let $A=\{a_i:i\in[m]\}, B=\{b_i:i\in[m]\}, D=\{d_i:i\in[m]\}$.  Define an edge coloring $\gamma$ by

				\begin{align*}
					\gamma(a_1 b_i)&=1\mathrm{\ if\ }i>1\\
					\gamma(b_1 d_i)&=2\mathrm{\ if\ }i>1\\
					\gamma(d_1 a_i)&=3\mathrm{\ if\ }i>1\\
					\gamma(u v)&=\varphi(u v)\mathrm{\ for\ all\ other\ }u,v\in V.\\
				\end{align*}

				It is easy to check that $\gamma$ has the structure described above with $m=1$.  We have essentially moved $m-1$ vertices from each of $A$, $B$, and $D$, to $E$.  Since $a_1, b_1,$ and $c_1$ each have monochromatic degree $n-2$, any 2-factor must have edges of colors 1,2, and 3 under the coloring $\gamma$, so if it had all colors under $\varphi$, it still does under $\gamma$.
			\end{proof}
		\end{claim}
		
			We remark that the coloring $\gamma$ with $m=1$ in Claim \ref{q=0_optimal-quasi-ordered} is $Z$-quasi-ordered with $\abs{Z}=3$. As we have shown, if there exists such an $R_0$-polychromatic coloring $\varphi$ with $m>1$, then there exists one with $m=1$.  However, if $m>1$ and $n>6$, a coloring $\varphi$ satisfying properties \eqref{C5one}--\eqref{C5eight} might not be $R_0$-polychromatic.  This is because if $E$ has no internal edges with color $1$, then any $2$-factor with a $2m$-cycle consisting of alternating vertices from $A$ and $B$ has no edge with color $1$.  However, the modified coloring $\gamma$ (with $m=1$) is an $R_0$-polychromatic coloring because then colors $1$, $2$, and $3$ must appear in any $2$-factor.

		\begin{claim}\label{max-vertex-q=1}
			If $q\geq 1$ then, up to relabelling colors, every max vertex is a $(1,2)$-max vertex or a $(2,1)$-max vertex.
			\begin{proof}\let\qed\relax
				As before, we assume $v$ is a $(1,2)$-max vertex, that $\varphi(uv)=2$ and that $H\in R_q$ (or $H\in C_q^*$) is such that $uv$ is the only edge of color 2.  We know that $u$ is a $(2,t)$-max vertex for some color $t$.  By way of contradiction, suppose $u$ is a $(2,3)$-max vertex.  Then we have the configuration of Figure \ref{fig:graphfigure}, with $\abs{A}=\abs{B}=\abs{D}=m$.  If $uw$ is also an edge of $H$ then $w\in D$, since otherwise $\varphi(uw)=2$.  Let $Q$ be the set of vertices not in $H$ (so $\abs{Q}=q>0$) and suppose $p\in Q$ but $p\not\in B$.  Then we can replace $u$ in $H$ with $p$ to get a 2-regular graph (cycle) with no edge of color 2.  Hence $Q\subseteq B$.  Orient the edges of $H$ to get a directed graph $H'$ where $\vv{uv}$ is an arc.  Since $\abs{B\setminus Q}<\abs{D}$, and every vertex in $D$ appears in $H'$, for some $d\in D$ and $e\not\in B$, $\vv{de}$ is an arc in $H'$.  Since $\varphi(du)=3$ and $\varphi(ev)=1$, when you twist $uv$ and $de$ you get a 2-regular graph (cycle) with no edge of color 2, a contradiction.  Hence every max-vertex is a $(1,2)$-max vertex or $(2,1)$-max vertex.

			\end{proof}
		\end{claim}

\begin{claim}\label{q>1-XnotEmpty}
	If $q=1$ then, up to relabelling colors, the vertex set can be positioned into $S,T,W$ such that
	\begin{enumerate}
		\item\label{q>1-first} $S$ is the set of all $(1,2)$-max vertices
		\item $T$ is the set of all $(2,1)$-max vertices
		\item $W$ has no max vertices
		\item All internal edges in $S$ and all edges from $S$ to $W$ have color 1; all internal edges in $T$ and all edges from $T$ to $W$ have color 2
		\item\label{q>1-last} The edges of color 1 between $S$ and $T$ form two disjoint copies of $K_{m,m}$, as do the edges of color 2 (so $\abs{S}=\abs{T}=2m$, where $n-m-1$ is the maximum monochromatic degree)
	\end{enumerate}
	\begin{proof}\let\qed\relax
		By Claim \ref{max-vertex-q=1}, if $q\geq 1$, then every max vertex is a $(1,2)$ or $(2,1)$-max vertex.
		
		Let $S$ be the set of all $(1,2)$-max vertices and $T$ be the set of all $(2,1)$-max vertices, with $\abs{S}=s$ and $\abs{T}=t$, $s\leq t$, and let m be the maximum monochromatic degree.  Let $W=V(G)\setminus (S\cup T)$ and let $B$ be the complete bipartite graph with vertex bipartition $S,T$ and edges colored as they are in $G$.  So each vertex of $B$ in $S$ is incident with $m$ edges of color 2 and $t-m$ edges of color 1, and each vertex of $B$ in $T$ is incident with $m$ edges of color 1 and $s-m$ edges of color 2.  All edges of $G$ within $S$ and between $S$ and $W$ have color 1 (otherwise there would be a $(2,1)$-max vertex not in $T$) and all edges within $T$ and between $T$ and $W$ have color 2.
		
		We note that the edges of color 1 in $B$ satisfy the conditions of Lemma \ref{disjoint_union}, so $B$ has a $2s$-cycle of edges of color 1 unless $s=t=2m$ and the edges of color 1 (and those of color 2) form two disjoint copies of $K_{m,m}$.
		
		Again, let $v\in S$ and $u\in T$ be such that $c(uv)=2$, and let $H\in C_q^*(n)$ (or $H\in R_q(n)$), $q\geq 1$, be such that $uv$ is the only edge of color 2.  If $uw$ is also an edge of $H$ then $w\in S$, because otherwise $c(uw)=2$.  Hence if $z$ is a vertex of $G$ not in $H$ then $z\in T$, because otherwise we can replace $u$ with $z$ in $H$ to get $H''\in C_q^*(n)$ (or $H''\in R_q(n)$) with no edge of color 2.  That means that if $Q$ is the set of vertices of $G$ not in $H$, then $Q\subseteq T$.  Since $uv$ is the only edge in $H$ with color 2, each vertex in $T\setminus Q$ is adjacent in $H$ to two vertices in $S$, so there are $2(t-q)$ edges in $H$ between $S$ and $T$, where $q=\abs{Q}\geq t-s$.
		
		Let $M$ be the subgraph of $H$ remaining when the $2(t-q)$ edges in $H$ between $S$ and $T$ have been removed (along with any remaining isolated vertices).  If $q=t-s$ then, since every edge in $H$ incident to a vertex in $T$ goes to $S$, either $H$ is a $2s$-cycle and $W=\emptyset$ (if $H\in C_q^*(n)$) or the union of the components of $H$ which have a vertex in $T$ is a 2-regular graph spanning $S$ and $s=t-q$ vertices in $T$.  In either case, since $s<t$, we can replace the components of $H$ which intersect $T$ with the $2s$-cycle of edges of color 1 promised by Theorem \ref{Jackson}, to get an $H''\in C_q^*(n)$ (or $H''\in R_q(n)$) with no edge of color 2.  Hence $q>t-s$.
		
		Each component of $M$ is a path with at least one edge, both endpoints in $S$ with interior points in $S$ or $W$.  If a component has $j>2$ vertices in $S$, we split it into $j-1$ paths which each have their endpoints in $S$ with all interior points in $W$.  If a vertex of $S$ is an interior point in a component then it is an endpoint of two of these paths.  The number of such paths is $\frac{2(s-(t-q))}{2}=s-(t-q)>0$.
		
		We denote the paths by $P_1,P_2,\ldots,P_r$ where $r=s-(t-q)$.  For each $i$ in $[r]$ where $P_i$ has more than 2 vertices, we remove the edges containing the two endpoints (which are both in $S$), leaving a path $W_i$ whose vertices are all in $W$ (the union of the vertices in all the $W_i$'s is equal to $W$).
		
		We will now show that there cannot be a $2s$-cycle of edges of color 1 in $B$.  Suppose $J$ is such a $2s$-cycle.  Let $R=\{x_1,x_2,\ldots,x_r\}$ be the set of any $r$ vertices in $T\cap V(J)$ and let $K$ be the subgraph of $J$ obtained by removing the $r$ vertices in $R$.  For each $i\in [r]$ let $y_{ia}$ and $y_{ib}$ be the vertices adjacent to $x_i$ in $J$.  Both are in $S$ and possibly $y_{ib}=y_{ja}$ if $i\neq j$.  Now, for each $i\in[r]$, attach $W_i$ to $y_{ia}$ and $y_{ib}$ ($R_i$ can be oriented either way).  More precisely, if $W_i$ is the path $w_{i1},w_{i2},\ldots,w_{id}$ in $W$, we attach it to $K$ by adding the edges $y_{ia}w_{i1}$ and $y_{ib}w_{id}$, while if $W_i$ is empty (meaning the $i^{\rm{th}}$ component of $M$ has only two vertices, so none in $W$) we add the edge $y_{ia}y_{ib}$.  The resulting graph $H''$ has no edge of color 2, since we constructed it using only edges from $J$ and edges from $H$ within $S\cup W$.  Since $V(H'')=V(G)\setminus R$, $H''$ has $n-q$ vertices.  Clearly $H''$ is 2-regular and, if $H$ is a cycle, so is $H''$ (if $H$ is not a cycle, $H''$ will still be a cycle if $H$ does not have any components completely contained in $W$).  Thus $H''\in R_q(n)$ ($H''\in C_q^*(n)$) and has no edge of color 2, a contradiction. Hence there is no $2s$-cycle of edges of color 1 in $B$.
		
		By Lemma \ref{disjoint_union} it follows that $s=t=2m$ with the edges of color 1 forming two vertex-disjoint copies of $K_{m,m}$.  (If these two disjoint copies have vertex sets $S_1\cup T_1$ and $S_2\cup T_2$, where $S_1\cup S_2=S$ and $T_1\cup T_2=T$, then $S_1\cup T_2$ and $S_2\cup T_1$ are the vertex sets which induce two disjoint copies of $K_{m,m}$ with edges of color 2.)  We have now verified that properties \eqref{q>1-first}--\eqref{q>1-last} hold if $q\geq 1$.  We will now show we get a contradiction if $q\geq 2$.
		
		Assume $q\geq 2$.  Let $T_1$ and $T_2$ be the sets of vertices in $T$ in the two $s$-cycles of edges of color 1 ($\abs{T_1}=\abs{T_2}=\frac{s}{2}$, $T_1\cup T_2=T$).  Recall that $v\in S$, $u\in T$, and $uv$ is the only edge of $H$ of color 2.  The subgraph $M$ of $H$ defined earlier still consists of paths which can be split into paths $P_1,P_2,\ldots,P_q$ (since $r=s-t+q=q$) with endpoints in $S$ and interior points in $W$.  Let $J$ be the union of the two $s$-cycles of edges of color 1.  Choose the subset $Q$ of size $q$ so that it has at least one vertex in each of $T_1$ and $T_2$, say $Q=\{x_1,x_2,\ldots,x_q\}$ where $x_1\in T_1$ and $x_q\in T_2$.  Again, let $K$ be the subgraph obtained from $J$ by removing the vertices in $Q$.  Then, as before, the paths $W_1,W_2,\ldots,W_q$ (perhaps some of them empty) can be stitched into $K$.  We attach $W_i$ to $y_{ia}$ and $y_{ib}$ if $i\in[2,q-1]$ (just adding the edge $y_{ia}y_{ib}$ if $W_i$ is empty).  We attach $W_1$ to $y_{1a}$ and $y_{qb}$ and $W_q$ to $y_{1b}$ and $y_{qa}$, creating an $(n-q)$-cycle if no component of $H$ is contained in $W$, and a 2-regular graph spanning $n-q$ vertices if $H$ has a component contained in $W$.  There is no edge of color 2 in this graph contradicting the assumption that if $q\geq 2$ and $\mathcal{H}\in\{R_q(n),C_q^*(n)\}$ then the maximum monochromatic degree in all optimal $\mathcal{H}$-polychromatic colorings is less than $n-1$.
	\end{proof}
\end{claim}

\begin{claim}
	If $\sH\in\{C_1^*,R_1\}$ and there exists an $\sH$-polychromatic coloring satisfying \eqref{q>1-first}--\eqref{q>1-last} in Claim \ref{q>1-XnotEmpty} with $m>1$, then there exists one with $m=1$, i.e. one that is $Z$-quasi-ordered with $\abs{Z}=4$.

	\begin{proof}
		Assume there is an $R_1$-polychromatic coloring ($C_1^*$-polychromatic coloring) $c$ with $q=1$ satisfying \eqref{q>1-first} -- \eqref{q>1-last} of Claim \ref{q>1-XnotEmpty} where $s=t>2$.  Let $v$ and $x$ be vertices in $S$ and $u$ and $y$ be vertices in $T$ such that $c(vu)=c(xy)=2$ and $c(xu)=c(vy)=1$.  Let $c'$ be the coloring obtained from $c$ by recoloring the following edges (perhaps they are recolored the same color they had under $c$):
		\begin{center}
			\begin{tabular}{lll}
				$c'(vp)=1$ & for all & $p\in T\setminus\{u,y\}$\\
				$c'(xp)=1$ & for all & $p\in T\setminus\{u,y\}$\\
				$c'(zu)=2$ & for all & $z\in S\setminus\{v,x\}$\\
				$c'(zy)=2$ & for all & $z\in S\setminus\{v,x\}$\\
				$c'(zp)=3$ & for all & $p\in T\setminus\{u,y\}$ and $z\in S\setminus\{v,x\}$
			\end{tabular}
		\end{center}
		
		Since all but one edge incident to $v$ and $x$ have color 1 under $c'$, certainly every $(n-1)$-cycle contains an edge of color 1.  Similarly for $u$ and $y$ and edges of color 2.  Every edge which was recolored had color 1 or 2 under $c$, so $c'$ must be a polychromatic coloring with the same number of colors.  It has the desired form with $\abs{S}=\abs{T}=2$, so, in fact, is $Z$-quasi-ordered with $Z=\{v,x,u,y\}$.
	\end{proof}
	
\end{claim}

		%

		We remark that a coloring $c$ satisfying properties \eqref{q>1-first}--\eqref{q>1-last} of Claim \ref{q>1-XnotEmpty} with $s=t>2$ is actually not $R_1$-polychromatic.  To see this, let $S_1 \cup T_1$ and $S_2\cup T_2$ be the vertex sets of the two copies of $K_{m,m}$ of edges of color 1 ($S_1\cup S_2 = S$, $T_1\cup T_2=T$) where $v\in S_1, u\in T_2$ and $uv$ is the only edge of color 2 in $H\in R_1$.  The subgraph $M$ of $H$ in the proof of Claim \ref{q>1-XnotEmpty} has only one component (since $s-(t-q)=1$), a path $d w_1 w_2\ldots w_e z$ where $d\in S_1$, $z\in T_1$, and $\{w_1,w_2,\ldots,w_e\}\subseteq W$.  To construct a 2-regular subgraph with no edges of color 2 spanning $n-1$ vertices, remove a vertex $x$ in $T_2$ from one of the two $s$-cycles of edges of color 1.  If $y_a$ and $y_b$ are the two vertices in $S_2$ adjacent to $x$ in the $s$-cycle, attach the path $w_1 w_2 \ldots w_e$ to $y_a$ and $y_b$ to get a 2-regular subgraph with no edge of color 2 spanning $n-1$ vertices.  However, this construction cannot be done when $m=1$, so in this case you do get an $R_1$-polychromatic coloring.

		\begin{lem}\label{XoToO}
			Let $\sH\in\{R_q(n),C^*_q(n)\}$.  
			\begin{enumerate}[label=\alph*)]
				\item Suppose for some $X\neq \emptyset$ there exists an optimal $X$-ordered $\sH$-polychromatic coloring of $K_n$.  Then there is one which is ordered.
				\item Suppose there exists an optimal $Z$-quasi-ordered $\sH$-polychromatic coloring of $K_n$.  Then there is one which is quasi-ordered
			\end{enumerate}
			\begin{proof}
				Among all such $\sH$-polychromatic colorings we assume $\varphi$ is one such that
				\begin{enumerate}[label=\alph*)]
					\item\label{XoToO-a} if $\varphi$ is $X$-ordered then $X$ has maximum possible size
					\item\label{XoToO-b} if $\varphi$ is $Z$-quasi-ordered then the restriction of $\varphi$ to $V(K_n)\setminus Z$ is $T$-ordered for the largest possible subset $T$ of $V(K_n)\setminus Z$.  In this case, we let $X=Z\cup T$ so $\varphi$ is nearly $X$-ordered (one or two edges could be recolored to make it $X$-ordered).
				\end{enumerate}
				For both \ref{XoToO-a} and \ref{XoToO-b} we assume that $\varphi$ is such that its restriction to $G_m=K_n[Y]$ has a vertex $v$ of maximum possible monochromatic degree in $G_m$, where $Y=V(K_n)\setminus X$, $\abs{Y}=m$, and the degree of $v$ in $G_m$ is $d<m-1$ (if $d=m-1$ then $\abs{X}$ is not maximal).
				
				Since $v$ has maximum monochromatic degree $d$ in $G_m$, by Lemma \ref{max-vertex} it is a $(1,2)$-max vertex in $G_m$, for some colors 1 and 2, and if $u\in Y$ is such that $\varphi(uv)=2$, then $u$ is a $(2,t)$-max vertex for some color $t$ (perhaps $t=1$).
				
				As before, let $y_1,y_2,\ldots,y_d$ be vertices in $Y$ such that $c(vy_i)=1$ for $i=1,2,\ldots,d$.  As before, let $H\in\sH$ be such that $uv$ is its only edge with color 2.  Let $H'$ be a cyclic orientation of the edges of $H$ such that $\vv{uv}$ is an arc, and let $w_i$ be the predecessor of $y_i$ in $H'$ for $i=1,2,\ldots,d$.  As shown before, $c(w_i v)=2$ for $i=1,2,\ldots,d$.
				
				Suppose there is an edge of $H$ which has one vertex in $X$ and one in $Y$.  Then there exist $w\in Y$ and $x\in X$ such that $\vv{wx}\in H'$.  Certainly $w$ is  not the predecessor in $H'$ of any $y_i$ in $Y$.  Since $\varphi$ is $X$-constant and $uv$ is the only edge of color 2 in $H$, $\varphi(xv)=\varphi(xw)\neq 2$.  Now twist $xw,uv$ in $H$.  Since $\varphi(xv)\neq 2$, we must have $\varphi(wu)=2$, so $u$ is incident in $G_m$ to at least $d+1$ vertices of color 2, a contradiction \Lightning. Hence $H$ cannot have an edge with one vertex in $X$ and one in $Y$.
		
				Now suppose $x\in X$ and $x\not\in H$.  If $\varphi(xv)=\varphi(xu)\neq 2$ then $H\setminus\{uv\}\cup\{ux,xv\}$ is an $(n-q+1)$-cycle with no edge of color 2, which is clearly impossible if $\sH=R_q(n)$, and is impossible if $\sH=C^*_q(n)$ by Theorem \ref{theorem:six}.  Hence $\varphi(xv)=\varphi(xu)=2$ for each $x\in X$.

				Since $u$ is a $(2,t)$-max vertex for some color $t\neq 2$, we can repeat the above argument with $u$ in place of $v$.  That shows that $\varphi(xv)=\varphi(xu)=t$ for each $x\in X$, which is clearly impossible.

				It remains to consider the possibility that $\sH=R_q(n)$ and $X$ is spanned by a union of cycles in $H$.  Suppose $xz$ is an edge of $H$ contained in $X$.  Then we can twist $xz$ and $uv$ to get another subgraph in $R_q$ and, unless either $x$ or $z$ has main color $2$, this subgraph has no edge of color $2$.  Hence at least half the vertices in $X$ have main color $2$ (and more than half would if $H$ had an odd component in $X$).

				The above argument can be repeated with $u$ in place of $v$.  If $u$ is a $(2,t)$-max-vertex then that would show that at least half the vertices in $X$ have main color $t\neq 2$.  So each vertex in $X$ has main color $2$ or $t$.  Since $\varphi$ is $X$-ordered or nearly $X$-ordered, some vertex $x\in X$ has monochromatic degree $n-2$ or $n-1$ and the main color of $x$ must be 2 or $t$. Assume it is $2$.  Then every cycle containing $x$ has an edge with color 2, contradicting the assumption that $H$ has only one edge with color 2.  Similarly, we get a contradiction if the main color of $x$ is $t$.  We have shown there is no vertex $v$ with monochromatic degree $d<m-1$, so $\varphi$ is ordered or quasi-ordered.
			\end{proof}
		\end{lem}

Now there is not much left to do to prove Theorems \ref{theorem:two}, \ref{theorem:three}, and \ref{theorem:four}.

\subsection{Proof of Theorem \ref{theorem:four}} 
\label{sec:proof_of_theorem_ref_theorem_four}

Theorem \ref{theorem:five} takes care of the case of $C_q$-polychromatic colorings when $q\geq 2$ and $n\in[2q+2,3q+2]$.  The smallest value of $n$ for which there is a simply-ordered $C_q$-polychromatic 2-coloring is $n=3q+3$ (the coloring $\varphi_{C_q}$ in Section \ref{subsec:_k_hc_polychromatic_coloring}).  Hence if $q\geq 2$ and $\pcq\leq2$ then there exists an optimal simply-ordered $C_q$-polychromatic coloring except if $n-q$ is odd and $n\in[2q+2,3q+2]$, or if $q=2$ and $n=5$ (the coloring of $K_5$ with two monochromatic 5-cycles has no monochromatic 3-cycle).  So we need only consider $\sH\in\{R_q(n),C^*_q(n)\}$ (when $q\geq 2$).  Since \ref{structure-one} is not satisfied in Lemma \ref{structurelemma}, there exists an optimal $\sH$-polychromatic coloring with maximum monochromatic degree $n-1$.  That means it is $X$-ordered, for some nonempty set $X$, so by Lemma \ref{XoToO} there exists an optimal $\sH$-polychromatic coloring which is ordered, and then, by Lemma \ref{O2SO}, one which is simply-ordered.\hfill\qedsymbol


\subsection{Proof of Theorem \ref{theorem:two}} 
\label{sec:proof_of_theorem_ref_theorem_two}

If $\sH\in\{R_0(n),C_0(n)\}$ then, by Lemma \ref{structurelemma}, there exists an optimal $\sH$-polychromatic coloring which is $Z$-quasi-ordered with $\abs{Z}=3$.  Then, by Lemma \ref{XoToO}, there exists one which is quasi-ordered and then, by Lemma \ref{O2SO}, one which is quasi-simply-ordered with $\abs{Z}=3$, so recoloring one edge would give a simply-ordered coloring.\qedsymbol


\subsection{Proof of Theorem \ref{theorem:three}} 
\label{sec:proof_of_theorem_ref_theorem_three}
	
	Exactly the same as the proof of Theorem \ref{theorem:two}, except now $\abs{Z}=4$, so two edges need to be recolored to get a simply-ordered coloring.

\section{Polychromatic cyclic Ramsey numbers} 
\label{sec:polychromatic_cyclic_ramsey_numbers}

Let $s, t$, and $j$ be integers with $t\geq 2, s\geq 3, s\geq t$, and $1\leq j\leq t-1$.  We define $\cyram(s,t,j)$ to be the smallest integer $n$ such that in any $t$-coloring of the edges of $K_n$ there exists an $s$-cycle that uses at most $j$ colors. Erd{\H{o}}s and Gy\'{a}rf\'{a}s \cite{Erdos:1997} defined a related function for cliques instead of cycles.  So $\cyram(s,t,1)$ is the classical $t$-color Ramsey number for $s$-cycles and $\cyram(s,2,1)=c(s)$, the function in Theorem \ref{FS}.  While it may be difficult to say much about the function $\cyram(s,t,j)$ in general, if $j=t-1$ we get $\cyram(s,t,t-1)=\pr_t(s)$ the smallest integer $n\geq s$ such that in any $t$-coloring of $K_n$ there exists an $s$-cycle that does not contain all $t$ colors.  This is the function of Theorem \ref{extension} if $t\geq 3$, while $\pr_2(s)=c(s)$.

\subsection{Proof of Theorem \ref{extension}} 
\label{sec:proof_of_theorem_extension}

	Let $q\geq 0, s\geq 3$, and $n$ be integers with $n=q+s$. Assume $q\geq 2$. By Theorem \ref{theorem:four} and the properties of the coloring $\varphi_{C_q}$ (see Section \ref{subsec:_k_hc_polychromatic_coloring}), there exists a $C_q$-polychromatic $t$-coloring of $K_n$ if and only if
	\begin{align*}
		q+s&=n\geq (2^t -1)q+2^{t-1}+1,\\
		s&\geq(2^t-2)q+2^{t-1}+1,\\
		q&\leq\frac{s-2^{t-1}-1}{2^t -2} = \frac{s-2}{2^t -2}-\frac{1}{2}
	\end{align*}
	Since $q\geq 2$, we want to choose $s$ so that the right-hand side of the last inequality is at least 2, so
	\begin{align*}
		s-2&\geq \frac{5}{2}(2^t -2) = 5\cdot 2^{t-1}-5\\
		s&\geq 5\cdot 2^{t-1}-3
	\end{align*}
	
	So if $s\geq 5\cdot 2^{t-1}-3$, then the smallest $n$ for which there does not exist a $C_q$-polychromatic $k$-coloring is $n=q+s$ where $q>\frac{s-2}{2^t-2}-\frac{1}{2}$, so $n=s+\left\lfloor \frac{s-2}{2^t-2} +\frac{1}{2} \right\rfloor=s+\round\left(\frac{s-2}{2^t-2}\right)$.
	
	We note that if $s\geq 5\cdot 2^{t-1}-3$ then $\round\left( \frac{s-2}{2^t-2} \right)\geq \round\left( \frac{5}{2} \right)=3$, so $\pr_t(s)\geq s+3$ if $s\geq 5\cdot 2^{t-1}-3$.
	
	Now we assume that $\pr_t(s)=s+2$.  So $s+2$ is the smallest value of $n$ for which in any $t$-coloring of the edges of $K_n$ there is an $s$-cycle which does not have all colors, which means there is a polychromatic $t$-coloring when $n=s+1$.  Since $q=1$ in such a coloring, by Theorem \ref{theorem:three} and the properties of the coloring $\varphi_{C_1}$, $n\geq 5\cdot 2^{t-2}$.  Hence if $s\in[5\cdot 2^{t-2}-1,5\cdot 2^{t-1}-4]$, then $\pr_t(s)=s+2$.
	
	Now we assume that $\pr_t(s)=s+1$.  So $n-s$ is the largest value of $n$ such that in any $t$-coloring of $K_n$, every $s$-cycle gets all colors.  So $q=n-s=0$ and, by Theorem \ref{theorem:two} and properties of the coloring $\varphi_{C_0}$, $n\geq 3\cdot 2^{t-3}+1$.
	
	Finally, since the $t$-coloring $\varphi_{C_0}$ requires $n\geq 3\cdot 2^{t-3}+1$ where $t\geq 4$ if $n\leq 3\cdot 2^{t-3}$ and $t\geq 4$, then in any $t$-coloring of $K_n$, some Hamiltonian cycle will not get all colors, so $\pr_t(s)=s$ if $3<s\leq3\cdot2^{t-3}$.


\section{Conjectures} 
\label{sec:conjectures_and_closing_remarks}

We mentioned that we have been unable to prove a result for 2-regular graphs analogous to Theorem \ref{theorem:six} for cycles.  In fact we think it even holds for two colors, except for a few cases with $j$ and $n$ small.

\begin{conj}\label{2-regular_conjecture}
	Let $n\geq 6$ and $j$ be integers such that $3\leq j <n$, and if $j=5$ then $n\geq 9$, and let $\varphi$ be an edge-coloring of $K_n$ so that every 2-regular subgraph spanning $j$ vertices gets all colors.  Then every 2-regular subgraph spanning at least $j$ vertices gets all colors under $\varphi$.
\end{conj}

This does not hold for $j=3, n=4$, and 3 colors; $n=5, j=3$, and 2 colors.

We can extend the notions of $Z$-quasi-ordered, quasi-ordered, and quasi-simply-ordered to sets $Z$ of larger size, allowing a main color to have degree less than $n-2$.  Let $q\geq 0$ and $r\geq 1$ be integers such that $q\leq 2r-3$.  Hence $\frac{2r-2}{q+1}\geq 1$, and we let $k=\left\lfloor \frac{2r-2}{q+1} \right\rfloor+1\geq 2$ and $z=k(q+1)$.  Let $Z$ be a set of $z$ vertices.  We define a \emph{seed-coloring} $\varphi$ with $k$ colors on the edges of the complete graph $K_z$ with vertex set $Z$ as follows.  Partition the $z$ vertices into $k$ sets $S_1,S_2,\ldots,S_k$ of size $q+1$.  For $j=1,2,\ldots,k$, all edges within $S_j$ have color $j$, all edges between $S_i$ and $S_j$ ($i\neq j$) have color $i$ or $j$, and for each $j$ and each vertex $v$ in $S_j$, $v$ is incident to $\left\lceil \frac{(q+1)(k-1)}{2} \right\rceil$ or $\left\lfloor \frac{(q+1)(k-1)}{2} \right\rfloor$ edges with colors other than $j$ (so, within round off, half of the edges from each vertex in $S_j$ to vertices in other parts have color $j$).  We say each vertex in $S_j$ has main color $j$.

If $n\geq z$, we get a $Z$-quasi-ordered coloring $c$ of $K_n$ which is an extension of the coloring $\varphi$ on $Z$ if for each $j$ and each $v\in S_j$, $c(vy)=j$ for each $y\in V(K_n)\setminus Z$.  If $c$ is $Z$-quasi-ordered then it is quasi-ordered if $c$ restricted to $V(K_n)\setminus Z$ is ordered, and quasi-simply-ordered if $c$ restricted to $V(K_n)\setminus Z$ is simply-ordered.


If $r>0$ and $q\geq 0$ are integers we let $\mathscr{R}(n,r,q)$ be the set of all $r$-regular subgraphs of $K_n$ spanning precisely $n-q$ vertices (assume $n-q$ is even if $r$ is odd, so the set is nonempty), and if $r\geq 2$ let $\mathscr{C}(n,r,q)$ be the set of all such subgraphs which are connected.  

Since $k-1=\left\lfloor \frac{2r-2}{q+1} \right\rfloor\leq \frac{2r-2}{q+1}$, we have $r\geq\frac{(q+1)(k-1)}{2}+1>\left\lceil \frac{(q+1)(k-1)}{2} \right\rceil$.  So if $H$ is in $\mathscr{R}(n,r,q)$ or $\mathscr{C}(n,r,q)$, then $H$ contains an edge with each of the $k$ colors on edges within $Z$, because it contains at least one vertex in $S_j$ for each $j$, and fewer than $r$ of the edges incident to this vertex have colors other than $j$.  We can get an $\mathscr{R}(n,r,q)$-polychromatic or $\mathscr{C}(n,r,q)$-polychromatic quasi-simply-ordered coloring of $K_n$ with $m>k$ colors by making the color classes $M_t$ on the vertices in $V(K_n)\setminus Z$ for $t=k+1,k+2,\ldots,m$ sufficiently large.  If $H\in\mathscr{R}(n,r,q)$, for each $t\in[k+1,m]$ we will need the size of $M_t$ to be at least $q+1$ more than the sum of the sizes of all previous color classes, while if $H\in\mathscr{C}(n,r,q)$ we will need the size of $M_t$ to be at least $q$ more than the sum of the sizes of all previous classes, with an extra vertex in $M_m$.  To try to get optimal polychromatic colorings we make the sizes of these color classes as small as possible, yet satisfying these conditions.

For example, if $r=2$ and $q=0$ then $k=\left\lfloor \frac{2r-2}{q+1} \right\rfloor +1=3$ and $z=k(q+1)=3$, and we get the quasi-simply-ordered colorings $\varphi_{R_0}$ and $\varphi_{C_0}$ with $\abs{Z}=3$ of Theorem \ref{theorem:two}.  If $r=2$ and $q=1$ then $k=2$ and $z=4$, and we get the colorings $\varphi_{R_1}$ and $\varphi_{C_1}$ with $\abs{Z}=4$ of Theorem \ref{theorem:three}.

\begin{example}[$r=3,q=0$, so $k=5,z=5$]\label{previous-example}
	Let $\varphi$ be the edge coloring obtained where $\{v_1,v_2,v_3,v_4,v_5\}=Z$ such that $v_i v_{i+1}$ and $v_i v_{i+2}$ ($\operatorname{mod} 5$) have color $i$.  The edges connecting $v_i$ to the remaining vertices in $V(K_n)\setminus Z$ are color $i$.  See Figure \ref{fig:kFrpoly}.
%

	\begin{figure}[htbp]
		\centering
		\begin{tikzpicture}[mystyle/.style={draw,shape=circle,fill=white}]
			\def\ngon{5}
			\node[regular polygon,regular polygon sides=\ngon,minimum size=3cm] (p) {};
			\foreach\x in {1,...,\ngon}{\node[mystyle] (p\x) at (p.corner \x){$v_{\x}$};}
			\foreach\i in {1,...,\ngon} {
				\foreach\t in {1,2}{
				\pgfmathsetmacro{\j}{int(mod(\i-1+\t,\ngon)+1)}
					\draw (p\i) --node {\i} (p\j);
				}
			}
		\foreach \i in {2,...,5}{
					\pgfmathsetmacro{\angle}{int(18+72*\i)}
					\draw (p\i) --node {\i} (\angle:2.5);
				}
			\draw (p1) --node[left] {1} (90:2.5);
			\path [draw=black,fill=gray,fill opacity=.25,even odd rule] (0,0) circle (3.5) (0,0) circle (2.5);
			\node (graphlabel) at (0,3) {$V(K_n)\setminus Z$};
		\end{tikzpicture}
		\caption{The coloring for Example \ref{previous-example}.}
		\label{fig:kFrpoly}
	\end{figure}
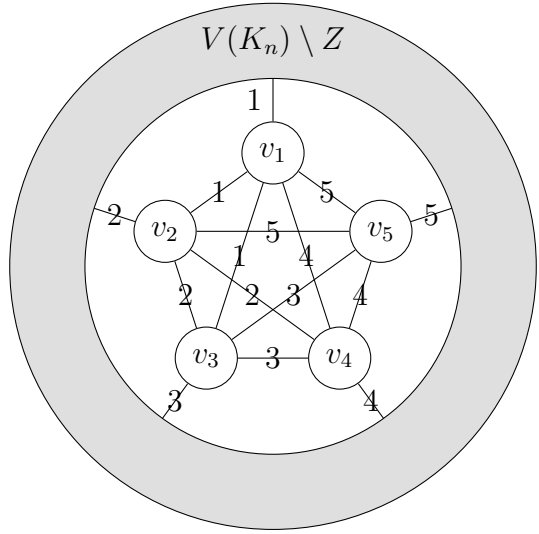
\end{example}

\begin{example}[$r=3,q=3,k=2,z=8$]
	$Z$ has two color classes, 4 vertices in each.  The complete bipartite graph between these two sets of vertices could have two vertex disjoint copies of $K_{2,2}$ of one color and also of the other color, or could have an 8-cycle of each color.
\end{example}

\begin{example}[$r=4,q=2,k=3,z=9$]
	So $S_1,S_2,S_3$ each have size $q+1=3$. One way to color the edges between parts is for $j=1,2,3$, each vertex in $S_j$ is incident with 2 edges of color $j$ to vertices in $S_{j+1}$ and 1 edge of color $j$ to a vertex in $S_{j-1}$ (so is incident with one edge of color $j+1$ and two edges of color $j-1$, cyclically).  The smallest value of $n$ for which this seed can generate a quasi-simply-ordered $\mathscr{R}(n,4,2)$-polychromatic coloring with 5-colors is $n=45$ (the 4$^{\rm{th}}$ and 5$^{\rm{th}}$ color classes would have sizes $9+2+1=12$ and $21+2+1=24$ respectively), while to get a simply-ordered $\mathscr{R}(n,4,2)$-polychromatic coloring with 5 colors you would need $n\geq 69$ (color class sizes $3,3,9,18,36$ works).
\end{example}

\begin{conj}
	Let $r\geq 1$ and $q\geq 0$ be integers such that $q\leq 2r-3$.  Let $k=\left\lfloor \frac{2r-2}{q+1} \right\rfloor+1\geq 2$ and $z=k(q+1)$.  If $n\geq z$ and $n-q$ is even if $r$ is odd, then there exist optimal quasi-simply-ordered $\mathscr{R}(n,r,q)$ and $\mathscr{C}(n,r,q)$-polychromatic colorings with seed $Z$ with parameters $r,q,k,z$.
\end{conj}

It is not hard to check that each of these quasi-simply-ordered colorings does at least as well as a simply-ordered coloring for those values of $r$ and $q$.  The only quesiton is whether some other coloring does better and the conjecture says no.

What if $\frac{2r-2}{q+1}<1$? Then $k=\left\lfloor \frac{2r-2}{q+1} \right\rfloor+1=1$, which seems to be saying no seed $Z$ exists with at least 2 colors.

\begin{conj}
	Let $r\geq 1$ and $q\geq 0$ be integers with $q\geq 2r-2$, $n\geq q+r+1$, and not both $r$ and $n-q$ are odd.  Then there exists an optimal simply-ordered $\mathscr{R}(n,r,q)$-polychromatic coloring of $K_n$.  If $r\geq 2$ then there exists a $\mathscr{C}(n,r,q)$-polychromatic coloring of $K_n$ (unless $r=2$, $q\geq 2$, $n-q$ is odd, and $n\in[2q+2,3q+1]$).
\end{conj}

Theorem \ref{theorem:one} says this conjecture is true for $r=1$.  Theorem \ref{theorem:four} says it is true for $\mathscr{C}(n,r,q)$ for $r=2$ and that it would be true for $\mathscr{R}(n,r,q)$ for $r=2$ if Theorem \ref{theorem:six} held for 2-regular graphs.

\bibliographystyle{plain}
\bibliography{Poly1HC2}
\end{document}